\documentclass[11pt]{amsart}
\usepackage{amsmath,amssymb,a4wide} 
\usepackage{xcolor}
\usepackage{mathrsfs}
\usepackage{dsfont}
\usepackage{mathabx}
\usepackage{tikz,pgfplots}
\usepackage{hyperref}
\usepackage{ulem}
\usepackage{graphicx}

\newtheorem{theorem}{Theorem}[section]
\newtheorem{lemma}{Lemma}[section]
\newtheorem{proposition}{Proposition}[section]
\newtheorem{remark}{Remark}[section]

\newtheorem{definition}{Definition}[section]
\newtheorem{example}{Example}[section]
\newtheorem{hypo}{Hypothesis}[section]

\newcommand{\Z}{{\mathbb Z}} 
\newcommand{\R}{{\mathbb R}} 
\newcommand{\C}{{\mathbb C}} 
\newcommand{\N}{{\mathbb N}} 
 
\newcommand{\T}{{\mathbb T}} 

\newcommand{\dd}{{\rm d}}
\newcommand{\Hs}{{H^\sigma}}
\newcommand{\Hss}{{(H^\sigma)^s}}

\usepackage{color}

\title[High order linearly implicit methods for semilinear evolution PDEs]
{High order linearly implicit methods for semilinear evolution PDEs}

\author[G. Dujardin]{Guillaume Dujardin}
\address[G. Dujardin]{Univ. Lille, Inria, CNRS, UMR 8524 - Laboratoire Paul Painlev\'e F-59000}
\author[I. Lacroix-Violet]{Ingrid Lacroix-Violet}
\address[I. Lacroix-Violet]{Universit{\'e} de Lorraine, CNRS, IECL, F-54000 Nancy, France}

\begin{document}


\begin{abstract}
  This paper considers the numerical integration of semilinear evolution PDEs using the high order linearly implicit methods developped in \cite{DL2020} in the ODE setting. These methods use a collocation Runge--Kutta method as a basis, and additional variables that are updated explicitly and make the implicit part of the collocation Runge--Kutta method only linearly implicit. In this paper, we introduce several notions of stability for the underlying Runge--Kutta methods as well as for the explicit step on the additional variables necessary to fit the context of evolution PDE. We prove a main theorem about the high order of convergence of these linearly implicit methods in this PDE setting, using the stability hypotheses introduced before. We use nonlinear Schr\"odinger equations and heat equations as main examples but our results extend beyond these two classes of evolution PDEs. We illustrate our main result numerically in dimensions 1 and 2, and we compare the efficiency of the linearly implicit methods with other methods from the litterature. We also illustrate numerically the necessity of the stability conditions of our main result.
\end{abstract}


\maketitle
%



\section{Introduction}\label{sec:intro}

%
%
%

High order linearly implicit methods for the time integration of evolution problems have been derived,
analysed and implemented recently in \cite{DL2020}.
In the context of evolution ODEs, \cite{DL2020} provides sufficient and constructive conditions
to achieve any arbitrarily high order with such methods.
The goal of this paper is to extend the design and analysis of such methods
to the case of semilinear evolution PDEs, that we semi-discretize in time.
We mostly focus on nonlinear Schr\"odinger equations (NLS)
and nonlinear heat equations (NLH),
even if the analysis can also be extended to other semilinear equations of the form
\begin{equation}
    \label{eq:general}
    \hspace{3.5cm}\partial_t u (t,{\bf x}) = Lu(t,{\bf x}) + N(u(t,{\bf x}))u(t,{\bf x}),
    \qquad \qquad \qquad \qquad
    t\geq 0, {\bf x}\in \Omega,
\end{equation}
where $L$ is an unbounded linear operator on some Banach space $X$ of functions on some $\Omega$
with appropriate boundary conditions
The function $N$ is a nonlinear function
of the unknown $u$ from $X$ to itself.
The unknown $u$ is a real or complex-valued function of time $t\geq 0$ and space ${\bf x}\in\Omega$.
The autonomous evolution equation \eqref{eq:general}
is supplemented with an initial condition at $t=0$.
More precise hypotheses and their corresponding
functional framework are introduced below.
For example, the cubic NLS may correspond to $N(u)=i q |u|^2$ (for some $q\in\R$)
and $L=i\Delta$, and the cubic NLH may correspond to $L=\Delta$ and $N(u)=\pm |u|^2$.
In classical cases depending on the geometry of $\Omega$ and the choice of boundary
conditions ({\it e.g.} $\Omega=\R^d$ or $\Omega=(\R/\Z)^d=\T^d$),
and the choice of $X$, the localization of the spectrum of $L$ is known.
Indeed, the spectrum of $L$ lies on the purely imaginary axis for NLS,
and in the left hand side of the complex plane for NLH.
In particular, we provide in this paper sufficient conditions to achieve any arbitrarily high order
with linearly implicit methods for the NLS and NLH equations.
Our results of course extend to other semilinear evolution PDEs.

The numerical integration of such problems has a long history, and
numerical schemes have been derived and analysed in several contexts.
For the NLS equation, let us mention for example \cite{DFP81} and \cite{DSS20}
where the Crank-Nicolson scheme is studied,
\cite{WH86} where the Lie--Trotter split step method is introduced,
\cite{Akrivis91} where some Runge--Kutta methods are used with Galerkin space discretization,
\cite{Besse2004} where a relaxation method is introduced,
which is proved to be of order $2$ in \cite{BDDLV19}.
Methods with higher order in time, for example exponential methods \cite{DUJ09,nous2017} or splitting
methods (see \cite{BBD02} \cite{Lubich08} for error estimates), have also been designed and studied.
The numerical behaviour in the semiclassical limit has also been studied \cite{Klein2007,BCM13}.
For the NLH equation, one may mention Crank--Nicolson \cite{Akrivis2006} as well as
splitting methods \cite{Descombes01,DM04,CCV09},
exponential methods \cite{HO05} \cite{HO10}, including Lawson and exponential Runge--Kutta methods.

Some of these methods are fully implicit ({\it e.g.} Crank--Nicolson),
some are fully explicit ({\it e.g.} a Lawson method based on an explicit Runge--Kutta method),
and some are linearly implicit ({\it e.g.} for the NLS equation, the relaxation scheme
of \cite{Besse2004}; for nonlinear parabolic equations,
the linearly implicit Rosenbrock methods and $W$-methods analysed in \cite{LO95},
the linearly implicit multistep methods analysed in \cite{Akrivis04},
the IMEX schemes analysed in \cite{Calvo2001}).
In some sense, the methods proposed and analysed in this paper appear as generalizations
of the relaxation method of \cite{Besse2004},
and form a new class (introduced in \cite{DL2020}) of linearly implicit methods,
different from the ones listed above.
Indeed, they are linearly implicit : at the cost of introducing extra variables,
they require only the solution of one linear system per time step
so that they do not require CFL conditions to ensure stability that usually appear in explicit methods.
Of course, the solution of one high dimensional linear system per time step has a computational cost,
when compared to other high order methods.
For the solution of such a system, one may rely on very efficient techniques,
either direct (LU factorization, Choleski factorization, etc) or iterative
(Jacobi method, Gauss--Seidel method, conjugate gradient, Krylov subspace method, etc \cite{Saad03}),
depending on the structure of the problem at hand.
Depending on the problem, linearly implicit methods can outperform high order
methods from the litterature.
In particular, the linearly implicit methods below can be competitive when the spectrum of the
linear operator $L$ is unknown (only the localization of the spectrum in the complex plane matters,
in some sense), while exponential methods may be faster for problems where that
spectrum is known. An example can be found in Section 3.2.2 of \cite{DL2020} for an NLS equation
on a domain where the spectrum of the Laplace operator is not known and a linearly implicit method
of order $2$ belonging to the class studied below outperforms the Strang splitting method.

Other authors have considered adding extra variables for the time integration of evolution problems.
For example, \cite{ShenXu18,ChengShen18} introduced scalar auxiliary variable methods (SAV)
and multiple scalar auxiliary variable methods (MSAV).
These methodes were introduced to produce unconditionnaly stable schemes
for dissipative problems with gradient flow structure.
The auxiliary variables in this context are used to ensure discrete energy decay.
The order of the methods (1 or 2 in the references above) was not the
main issue. In contrast, the linearly implicit methods analysed in this paper achieve
arbitratily high order in time.

The reader may refer to the Introduction section of \cite{DL2020} for the comparision of the linearly
implicit methods introduced there and analysed here in a PDE context, and other classical time
integration methods from the litterature (one-step methods, including Runge--Kutta methods,
linear multistep methods, composition methods, {\it etc}).

The linearly implicit methods of \cite{DL2020},
that we further analyse below for NLS and NLH equations,
embed a classical collocation Runge--Kutta method with
additional variables, in the spirit of the relaxation method introduced by Besse \cite{Besse2004}.
The analysis of the convergence of these linearly implicit methods applied
to semilinear problems like NLS and NLH relies on stability hypotheses of two different kinds :
on the classical collocation Runge--Kutta method used on the one hand,
and on the explicit update formula for the additional variables.
For the stability of the Runge--Kutta method itself, we refer to the work of
Crouzeix and Raviart \cite{Crouzeix1974,Crouzeix80}.
In this paper, we generalize some stability
notions defined in these works (see Definition \ref{def:stab} below and also \cite{DL2023RK}),
that allow to ensure global stability of our linearly implicit methods.
For the stability of the explicit update of the additional
variables, we extend results of \cite{DL2020} in the EDO case to the PDE \eqref{eq:general}.

This paper is organized as follows.
Section \ref{sec:descriptionmethodes} describes the extension of the linearly implicit
methods of \cite{DL2020} to the PDE setting.
First, we adapt in Section \ref{subsec:underlyingRK} the classical Runge--Kutta framework for ODEs
to the PDE context. This allows for a description of the linearly implicit methods in the PDE
setting in Section \ref{subsec:linimplmeth}. In Section \ref{subsec:functionalframework}, we introduce
the functional framework that allows to ensure the stability of the explicit
update of the additional variables (Lemma \ref{lem:normeD}).
In Section \ref{sec:CVlinimpl}, we analyse the consistency and convergence errors
of the linearly implicit methods in Section \ref{subsec:consistencyandconvergenceerrors},
and the stability of the linear Runge--Kutta step in Section \ref{subsec:analysisofthequasiRKstep}.
This allows for exposing our hypotheses on the exact solution of \eqref{eq:general}
as well as on the numerical parameters in Section \ref{subsec:hypoandresult}
to ensure in our main result (Theorem \ref{th:convergence})
that linearly implicit methods actually converge with high order.
The proof of the theorem is provided in Section \ref{subsec:proof}.
Section \ref{sec:massconservation} is devoted to proving that
for the NLS equation, the Cooper condition on the coefficients of the underlying collocation
Runge--Kutta method ensures that the mass is preserved by the linearly implicit numerical method.
We provide numerical experiments in Section \ref{sec:num}.
We consider NLS equations in 1D (methods of order 2 in Section \ref{subsubsec:methodesordre21D},
methods of order 4 in Section \ref{subsubsec:methodesordre41D} and a method of ordre 5 in Section
\ref{subsubsec:methodesordre51D}) and in 2D on a star-shaped domain
(see Section \ref{subsec:numNLS2D}).
We consider NLH equations in 1D (methods of order 2 in Section \ref{subsec:numNLH1D}).
In all these cases, we illustrate numerically the main result of the paper
(Theorem \ref{th:convergence}) and we compare the precision and efficiency of the
high order linearly implicit methods with that of other methods from the litterature.
We also demonstrate numerically in Section \ref{sec:num} the necessity of our stability
hypotheses in the PDE setting.
Section \ref{sec:conclusion} conludes this paper and introduces future works.
An appendix is proposed in Section \ref{sec:appendix}, which includes proofs of several
technical results (Sections \ref{subsec:proofASIstable} and \ref{subsec:techlemmas})
as well as Butcher's tableaux of the underlying Runge--Kutta methods used in this paper
for the numerical experiments of Section \ref{sec:num} to allow for reproducibility of the numerical
results.


\section{Linearly implicit methods for semilinear evolution PDEs}
\label{sec:descriptionmethodes}

In this section, we introduce the linearly implicit methods from \cite{DL2020} adapted
to the PDE setting of the equation \eqref{eq:general} as well as the associated functional
framework.

\subsection{Description of the linearly implicit methods}

Linearly implicit method from \cite{DL2020} applied to a semilinear PDE of the form
\eqref{eq:general} rely on a classical Runge--Kutta method,
referred to as the ``underlying Runger--Kutta
method'' in the following, as well as on additional variables that are explicitly updated at every
time step.
Before adapting the definition of the linearly implicit methods from \cite{DL2020} to the PDE
context at hand in \ref{subsec:linimplmeth}, we recall some basic notations for classical
Runge--Kutta methods applied to PDE problems and introduce
their (linear) stability function in Section \ref{subsec:underlyingRK}.

\subsubsection{The underlying Runge--Kutta method}
\label{subsec:underlyingRK}
We fix a number of stages $s\in\N^\star$ and we consider throughout this paper
a Runge-Kutta method with $s\in\N^\star$ stages, and real coefficients
$(a_{ij})_{1\leq i,j\leq s}$, $(b_i)_{1\leq i\leq s}$, $(c_i)_{1\leq i\leq s}$.
We group these coefficients in a square matrix $A$ of size $s$ with coefficients
$(a_{i,j})_{1\leq i,j\leq s}$, a column vector $b$ of size $s$
with coefficients $(b_i)_{1\leq i\leq s}$,
and a column vector $c$ of size $s$ with coefficients $(c_i)_{1\leq i\leq s}$.
We moreover denote by $\mathds{1}$ the column vector of size $s$ with all components equal to $1$.
We recall that the (linear) stability function of the method is defined for $\lambda\in\C$ by
\begin{equation}
\label{eq:defR}
  R(\lambda) = 1 + \lambda b^t (I-\lambda A)^{-1} \mathds{1},
\end{equation}
where $I$ is the identity matrix of size $s$.
Note that the function $R$ is a rational function of $\lambda$.
We denote by $A\otimes L$ the operator-valued matrix
\begin{equation*}
  A\otimes L = 
  \begin{pmatrix}
  a_{11} L  & \cdots & a_{1s} L\\
  \vdots & & \vdots \\
  a_{s1} L & \cdots & a_{ss} L
  \end{pmatrix}.
\end{equation*}
Similarly, we set
\begin{equation*}
  b^t\otimes L = (b_1 L , \cdots, b_s L).
\end{equation*}
The notation $I$ will be used for the identity operator in general.

\subsubsection{Linearly implicit methods}
\label{subsec:linimplmeth}
Following the methods introduced in an ODE context in \cite{DL2020},
we assume we are given an $s\times s$ real or complex matrix $D$ and a real or complex column vector
$\Theta=(\theta_1,\cdots,\theta_s)^t$, and
we consider some $h>0$ as a time step. Since the problem \eqref{eq:general} is autonomous,
we always consider $t=0$ as the starting time without loss of generality.
We take $u^0\in H^\sigma$ (classical Sobolev space of order $\sigma$, to be defined in the
next section), and we denote by $t\mapsto u(t)$ the exact solution
of \eqref{eq:general} with initial datum $u(0)=u^0$. Of course, for fixed $t\geq 0$, the function
$u(t)$ is itself a function of ${\bf x}\in\Omega$. We will often omit this dependency in the notations.
We initialize the linearly implicit method with $u_0\in H^\sigma$, close to $u^0$.
We initialize the additional variables $(\gamma_{-1,i})_{1\leq i\leq s}$ with approximations
of $(N(u((c_i-1)h)))_{1\leq i\leq s}$ and we set
$\Gamma_{-1}=(\gamma_{-1,i})^t_{1\leq i\leq s} \in (H^\sigma)^s$.
Assuming one knows an approximation $u_n$ of the exact solution $u$ at time $t_n=nh$
and approximations $\Gamma_{n-1}=(\gamma_{n-1,i})_{1\leq i\leq s}$
of $(N(u(t_{n-1}+c_i h)))_{1\leq i\leq n}$ for some integer $n\geq 0$,
the first step of the linearly implicit method consists in computing pointwise in ${\bf x}$
\begin{equation}
  \label{eq:heritagegamma}
  \Gamma_n = D \Gamma_{n-1} + \Theta N(u_n).
\end{equation}
Let us recall a definition from \cite{DL2020}.
\begin{definition}
The step \eqref{eq:heritagegamma} is said to be strongly stable
when the spectral radius $\rho(D)$ of the matrix $D$ is stricly smaller than $1$.
It is said to be consistent of order $s$ when it satisfies
\begin{equation}
  \label{eq:heritagegammaorders}
  V_c = D V_{c-\mathds{1}} + \vartheta,
\end{equation}
where $V_c$ is the Vandermonde matrix at points $(c_1,\dots,c_s)$
({\it i.e.} $(V_c)_{i,j}=c_i^{j-1}$ for all $i,j$),
$V_{c-\mathds{1}}$ is the Vandermonde matrix at points $(c_1-1,\dots,c_s-1)$
({\it i.e.} $(V_{c-\mathds{1}})_{i,j}=(c_i-1)^{j-1}$ for all $i,j$), and $\vartheta$
is an $s\times s$ matrix with $\Theta$ as first column and zero everywhere else.
\end{definition}
Existence of matrices $D$ and vectors $\Theta$,
such that the step \eqref{eq:heritagegamma} is strongly stable and of order $s$
is discussed in \cite{DL2020}, and formulas are provided as well.
In this paper, we will always assume that the step \eqref{eq:heritagegamma} is strongly stable
and is consistent of order $s$.
The second step of the linearly implicit method uses the Runge--Kutta method
introduced in \ref{subsec:underlyingRK} in the following way.
One first solves for $(u_{n,i})_{1\leq i\leq s}$ the following $s\times s$ linear system
\begin{equation}
  \label{eq:systlin}
  u_{n,i} = u_n + h\sum_{j=1}^s a_{ij}(L+\gamma_{n,j}) u_{n,j}
  \qquad 1\leq i \leq s.
\end{equation}
Setting $U_n\in (H^\sigma)^s$ as the unknown vector with components $(u_{n,i})_{1\leq i\leq s}$,
the system \eqref{eq:systlin} reads
\begin{equation}
  \label{eq:systlinvect}
  (I-hA\otimes L) U_n = u_n \mathds{1} + h A (\Gamma_n\bullet U_n),
\end{equation}
where $\Gamma_n\bullet U_n$ denotes the column vector with components
$(\gamma_{n,i} u_{n,i})_{1\leq i\leq s}$.
Afterwards, the third and last step of the linearly implicit method consists in computing explicitly
\begin{equation}
  \label{eq:RKexpl}
  u_{n+1} = u_n + h \sum_{i=1}^s b_i (L+\gamma_{n,i})u_{n,i}.
\end{equation}
This last step can be also written
\begin{equation}
  \label{eq:RKexplvect}
  u_{n+1} = u_n + h \left(b^t\otimes L\right) U_n +h b^t (\Gamma_n\bullet U_n).
\end{equation}

\begin{definition}[collocation methods]
  \label{def:colmeth}
  The Runge--Kutta method defined by $A$, $b$ and $c$ is said to be of collocation
  when $c_1,c_2,\dots,c_s$ are distinct real numbers
  and $A$ and $b$ satisfy for all $(i,j)\in\{1,\dots,s\}^2$,
  \begin{equation}
    \label{eq:defaijbi}
    a_{ij} = \int_0^{c_i} \ell_j(\tau) \dd \tau
    \qquad\text{and}\qquad
    b_i = \int_0^1 \ell_i(\tau)\dd \tau,
  \end{equation}
  where $(\ell_i)_{1\leq i\leq s}$ are the $s$ Lagrange polynomials associated to $(c_1,\dots,c_s)$
  defined by
  \begin{equation}
    \label{eq:defli}
    \forall i\in\{1,\dots,s\},\qquad
    \ell_i(\tau) = \prod_{\substack{j\neq i\\ j=1} }^s \frac{(\tau-c_j)}{(c_i-c_j)}. 
  \end{equation}
\end{definition}

\begin{remark}
  In the examples of this paper, we shall order the coefficients $c_1,\dots,c_s$ in such a way
  that $c_1<\cdots<c_s$ and assume that they all lie in $[0,1]$.
\end{remark}

\subsection{Functional framework}
\label{subsec:functionalframework}
In this paper, for expository purposes, we consider the case $\Omega=\R^d$ in particular in
the proofs of our results.
In this case, we denote by $H^\sigma$ the space
of functions $f$ from $\R^d$ to $\C$ such that
\begin{equation*}
  \int_{\R^d} |\hat f(\xi)|^2 (1+|\xi|^2)^\sigma \dd \xi <+\infty,
\end{equation*}
equipped with the norm
$\|f\|_{H^\sigma}=\left(\int_{\R^d} |\hat f(\xi)|^2 (1+|\xi|^2)^\sigma \dd \xi\right)^\frac12$,
where $\hat f$ stands for the Fourier transform of the function $f$.
The corresponding norm on $(H^\sigma)^s$ is defined for $F=(f_1,\dots,f_s)^t \in (H^\sigma)^s$ by
\begin{equation}
  \label{eq:norme1}
  \|F\|_{(H^\sigma)^s} = \left(\sum_{p=1}^s
  \int_{\R^d} |\hat f_p(\xi)|^2 (1+|\xi|^2)^\sigma \dd \xi\right)^\frac12.
\end{equation}

Of course, our results extend to several other cases. Let us mention two of them :
the $d$-dimensional torus $\Omega=\T^d$ and the case of a bounded open set $\Omega$ with
a Lipschitz boundary and homogeneous Dirichlet boundary conditions.
First, in the case of the torus $\T^d$, we denote by $H^\sigma$ the space
of functions $f$ from $\T^d$ to $\C$ such that
\begin{equation*}
  \sum_{k\in\Z^d} |\hat f(k)|^2 (1+|k|^2)^\sigma  <+\infty,
\end{equation*}
equipped with the norm
$\|f\|_{H^\sigma}=\left(\sum_{k\in\Z^d} |\hat f(k)|^2 (1+|k|^2)^\sigma  \right)^\frac12$,
where $\hat f$ stands for the sequence of Fourier coefficients of $f$.
The corresponding norm on $(H^\sigma)^s$ is defined for $F=(f_1,\dots,f_s)^t \in (H^\sigma)^s$ by
\begin{equation}
  \label{eq:norme2}
  \|F\|_{(H^\sigma)^s} = \left(\sum_{p=1}^s\sum_{k\in\Z^d} |\hat f_p(k)|^2 (1+|k|^2)^\sigma  \right)^\frac12.
\end{equation}
Second, in the case of a bounded open set $\Omega$ with a Lipschitz boundary, the Laplace
operator with homogeneous Dirichlet boundary conditions has compact inverse. Hence its
spectrum consists in a countable set of negative eigenvalues $(-\lambda_k^2)_{k\in\N}$
with $\lambda_k>0$ and $\lambda_k\underset{k\to +\infty}{\longrightarrow} +\infty$,
and there exists an orthogonal Hilbert basis $(e_k)_{k\in\N}$ of $L^2(\Omega)$
consisting in eigenvectors of this Laplace operator. Indeed, the integration by parts
formula holds true in this case (see Theorem 2.4.1 in \cite{Ledret}) and this allows to perform
the usual variational analysis (see Theorem 7.3.2 in \cite{Allaire}).
We denote by $H^\sigma$ the space of functions $f\in L^2(\Omega,\C)$ such that
\begin{equation*}
  \sum_{k\in\N} |\alpha_k|^2 (1+\lambda_k^2)^\sigma  <+\infty,
\end{equation*}
where $f=\sum_{k=0}^{+\infty} \alpha_k e_k$. Of course, the associated norm on $H^\sigma$
is defined by
\begin{equation*}
  \|f\|_{H^\sigma}=\left(\sum_{k\in\N} |\alpha_k|^2 (1+\lambda_k^2)^\sigma  \right)^\frac12,
\end{equation*}
and that on $(H^\sigma)^s$ is defined for $F=(f_1,\dots,f_s)^t \in (H^\sigma)^s$ by
\begin{equation}
  \label{eq:norme3}
  \|F\|_{(H^\sigma)^s} = \left(\sum_{p=1}^s\sum_{k\in\N} |\alpha^p_k|^2 (1+\lambda_k^2)^\sigma  \right)^\frac12,
\end{equation}
where for all $p\in\{1,\dots,s\}$, $f_p=\sum_{k=0}^{+\infty} \alpha^p_k e_k$.

In all cases, we chose $\sigma>d/2$ so that $H^\sigma$ is an algebra.
This is well known for $\Omega=\R^d$ and $\Omega=\T^d$, and holds for
$\Omega$ bounded open set of $\R^d$ with Lipschitz boundary (see for example after
Theorem 1.4.4.2 in \cite{Grisvard}).
Of course, one can even think of more general settings allowing
for the algebra property on $H^\sigma$ (see for example \cite{Fred12} and references therein),
but we will not do so in this paper.

The norm of linear continuous operators from $H^\sigma$ to itself is denoted by
$\|\cdot\|_{H^\sigma\to H^\sigma}$, and that of $(H^\sigma)^s$ to itself is denoted by
$\|\cdot\|_{(H^\sigma)^s\to (H^\sigma)^s}$.

\begin{lemma}
  \label{lem:normeD}
  As above, assume $D$ is a complex $s\times s$ matrix. The linear mapping
  $F\mapsto D F$ sends $(H^\sigma)^s$ to itself continuously.
  Moreover, if the step \eqref{eq:heritagegamma} is strongly stable ({\it i.e.} $\rho(D)<1$),
  then for all $\delta\in (\rho(D),1)$, there exists
  a norm $\|\cdot\|_{(H^\sigma)^s,D}$ on $(H^\sigma)^s$,
  equivalent to the usual norm defined above (in \ref{eq:norme1}, \ref{eq:norme2} or \ref{eq:norme3}),
  such that
  \begin{equation}
    \label{eq:inegnormeD}
    \forall F\in (H^\sigma)^s,\qquad
    \|DF\|_{(H^\sigma)^s,D} \leq \delta \|F\|_{(H^\sigma)^s,D}.
  \end{equation}
\end{lemma}

\begin{proof}
  We prove the result in the case of the full space $\R^d$. The proof in the other cases
  follows the same lines. First, one can check that
  \begin{equation}
    \label{eq:Dindependantdex}
    \forall F\in (H^\sigma)^s,\quad \forall \xi \in\R^d,\qquad \widehat{DF}(\xi) = D \hat F(\xi).
  \end{equation}
  Therefore,
  \begin{equation*}
    \forall F\in (H^\sigma)^s,\qquad \|DF\|_{(H^\sigma)^s} \leq C \|F\|_{(H^\sigma)^s},
  \end{equation*}
  where $C$ is any constant greater or equal to the norm of the matrix $D$ as a continuous linear
  operator from the hermitian space $\C^s$ to itself.
  If we assume moreover that $\rho(D)<1$ and $\delta\in(\rho(D),1)$ is given, then one
  can chose $\varepsilon>0$ sufficiently small to ensure that $\rho(D)+\varepsilon\leq \delta$.
  Denoting by $|\cdot|_D$ the norm on $\C^s$ provided by Lemma \ref{lem:normeinduite}, we can
  define a norm on $(H^\sigma)^s$ by setting
  \begin{equation*}
    \forall F\in \Hss,\qquad \|F\|_{\Hss,D}
    = \left(\int_{\R^d}|\hat F(\xi)|_D^2(1+|\xi|^2)^\sigma\dd\xi\right)^\frac12.
  \end{equation*}
  Since $|\cdot|_D$ is equivalent to the usual hermitian norm on $\C^s$, the norm
  $\|\cdot\|_{\Hss,D}$ is equivalent to the norm $\|\cdot\|_{\Hss}$.
  Moreover, thanks to \eqref{eq:Dindependantdex} and Lemma \ref{lem:normeinduite}, one has
  \begin{eqnarray*}
    \forall F\in \Hss,\qquad \|DF\|_{\Hss,D}^2
    & = & \int_{\R^d}|D \hat{F}(\xi)|_D^2(1+|\xi|^2)^\sigma\dd\xi\\
    & \leq & \int_{\R^d}(\rho(D)+\varepsilon)^2|\hat F(\xi)|_D^2(1+|\xi|^2)^\sigma\dd\xi\\
    & \leq & \delta^2 \|F\|_{\Hss,D}^2.
  \end{eqnarray*}
  This proves \eqref{eq:inegnormeD}.
\end{proof}

\begin{remark}
Let us define $\|\cdot\|_{\Hss,\infty}$ the norm on $\Hss$ defined as the maximum of the $\Hs$-norm of the components of the vectors. One can check easily that
$$\forall F \in \Hss, \qquad  \|F\|_{\Hss,\infty} \leq \|F\|_{\Hss} \leq \sqrt{s} \|F\|_{\Hss,\infty}.$$
In particular, this norm is also equivalent to the norm $\|\cdot\|_{\Hss,D}$.
\end{remark}


\section{Convergence
  for linearly implicit methods for NLS and NLH}
\label{sec:CVlinimpl}

In this section, we first introduce the consistency and convergence errors
of the linearly implicit method \eqref{eq:heritagegamma}-\eqref{eq:systlinvect}-\eqref{eq:RKexplvect}
in Section \ref{subsec:consistencyandconvergenceerrors}.
Then, we focus on the anlysis of the quasi-Runge--Kutta step \eqref{eq:systlinvect} in Section
\ref{subsec:analysisofthequasiRKstep}.
In Section \ref{subsec:hypoandresult}, we set the precise hypotheses on the exact solution of
\eqref{eq:general} as well as on the linearly implicit method
and we state the main result of the paper, which ensures that linearly implicit methods
have high order.
The proof of this main result is provided in Section \ref{subsec:proof}.

In this section, we assume $\Omega=\R^d$ or $\Omega=\T^d$ even if the results extend to more complex geometries (see Section \ref{subsec:functionalframework}). 

\subsection{Consistency and convergence errors}
\label{subsec:consistencyandconvergenceerrors}
\subsubsection{Consistency errors}
\label{ssec:consistencyerrors}
We assume that the equation \eqref{eq:general} with initial datum $u^0\in \Hs$ at $t=0$
has a unique smooth solution at least on some interval of the form
$(T_\star,T^\star)$ with $T_\star<0<T^\star$. We estimate consistency errors.
Since consistency errors rely on Taylor expansions of the exact solution with
respect to time, the analysis is very similar to that of \cite{DL2020}.
For all integer $n$ greater or equal to $-1$,
we denote by $t_n$ the time $nh$, where $h>0$ is the time step that we assume
sufficiently small to ensure that $T_\star<t_{-1}$.
\begin{definition}
  For all $n\geq 0$ and all $h>0$ small enough such that $t_{n+1}<T^\star$, we define
  the consistency error $R_{n}^{1}$ in $\Hss$ of the step \eqref{eq:heritagegamma}
  by setting
   \begin{equation}
   \label{eq:Rn1}
   R_n^1=\left [
       \begin{matrix}
         N(u(t_n+c_1 h)) \\
         \vdots\\
         N(u(t_n+c_s h))
       \end{matrix}
     \right ]
   -
   D
   \left [
   \begin{matrix}
         N(u(t_{n-1}+c_1 h)) \\
         \vdots\\
         N(u(t_{n-1}+c_s h))
       \end{matrix}
     \right ]
   -
   N(u(t_n))
   \left [
       \begin{matrix}
       \theta_1\\
       \vdots\\
       \theta_s
       \end{matrix}
     \right ].
 \end{equation}
 Similarly, we define the consistency error $R_n^2$ of step \eqref{eq:systlin}
 as the vector of $\Hss$ with entry number $i$ equal to
 \begin{equation}
   \label{eq:Rn2}
   \left(R_n^2\right)_i=u(t_n+c_ih)-u(t_n)-h\sum_{j=1}^s a_{ij} \left(L+N(u(t_n+c_j h))\right)u(t_n+c_j h).
 \end{equation}
 Moreover, we define the consistency error $R_n^3$ in $\Hs$ of step \eqref{eq:RKexpl}
 by setting
  \begin{equation}
   \label{eq:Rn3}
  R_n^3=u(t_{n+1})-u(t_n)-h\sum_{i=1}^s b_i \left(L+N(u(t_n+c_i h))\right)u(t_n+c_i h).
 \end{equation}
\end{definition}

We fix some final time $T\in (0,T^\star)$.

\begin{proposition}[consistency error of step \eqref{eq:heritagegamma}]
  Assume that the exact solution and the nonlinear term are such that
  $t\longmapsto N\circ u (t)$ has $s$ continuous derivatives in $\Hs$ on $[0,T]$.
  Assume moreover that the step \eqref{eq:heritagegamma} is consistent of order $s$.
  There exists a constant $C>0$ such that for all $n\in\N$ and $h>0$
  sufficiently small, such that $t_{n+1}\leq T$, one has
  \begin{equation}
    \label{eq:estimRn1}
    \| R_n^1 \|_{\Hss} \leq C h^s.
  \end{equation}
\end{proposition}

\begin{proof}
  The proof follows the very same lines as that of Lemma 1 in \cite{DL2020}.
\end{proof}

\begin{proposition}[consistency error of steps
  \eqref{eq:systlinvect}-\eqref{eq:RKexplvect}]
  Assume that the Runge-Kutta method defined in \ref{subsec:underlyingRK}
  is a collocation method of order $s$.
  Assume that the exact solution and the nonlinear term are such that
  $t\longmapsto u (t)$ has $s+2$ continuous derivatives in $\Hs$ on $[0,T]$.
  There exists a constant $C>0$ such that for all $n\in\N$ and $h>0$
  sufficiently small, such that $t_{n+1}\leq T$, one has
  \begin{equation}
    \label{eq:estimRn23}
    \| R_n^2 \|_{\Hss} \leq C h^{s+1} \qquad \text{and}\qquad
    \| R_n^3 \|_{\Hs} \leq C h^{s+1}.
  \end{equation}
\end{proposition}

\begin{proof}
  For all $n\geq 0$ and $h>0$ such that $t_{n+1}\leq T$, the errors
  $((R_n^2)_i)_{1\leq i\leq s}$ and $(R_n^3)$ are that of the quadrature method (corresponding
  to the Runge--Kutta collocation method) applied to the function $u'$
  at points $(t_n+c_jh)_{1\leq j\leq s}$ with weights
  $(h a_{i,j})_{1\leq j\leq s}$ on $[t_n,t_n+c_i h]$ for $(R_n^2)_i$ for all $i\in\{1,\dots,s\}$ and
  with weights $(h b_i)_{1\leq i\leq s}$ on $[t_n,t_{n+1}]$ for $R^3_n$.
  Therefore, they can be expressed,
  for example, with the Peano Kernel. They depend on the method, and on the exact solution,
  but not on the numerical solution.
  So, since the function $u'\in\mathcal C^{s+1}([0,T],\Hs)$, they are controlled
  as in \eqref{eq:estimRn23}.
\end{proof}

\subsubsection{Convergence errors}
We introduce the notation for the convergence errors that we use in the proof of convergence.
Following the notation in the paper \cite{DL2020}, we denote for all $h>0$ small enough
to ensure that the method \eqref{eq:heritagegamma}-\eqref{eq:systlin}-\eqref{eq:RKexpl}
is well-defined, and all $n\geq 0$ such that $t_{n+1}\leq T$ the following convergence errors:
$P_n\in \Hss$ is the vector with component number $i$ equal
to\break $(P_n)_i=N(u(t_n+c_ih))-\gamma_{n,i}$,
$Q_n\in \Hss$ is the vector with component number $i$ equal to $(Q_n)_i=u(t_n+c_ih)-u_{n,i}$, and
$e_n\in \Hs$ is defined by $u(t_n)-u_n$. Moreover, we set $z_n=\max_{0\leq k\leq n}(\|e_k\|_\Hs)$.

\subsection{Analysis of the linear quasi-Runge--Kutta step \eqref{eq:systlin}}
\label{subsec:analysisofthequasiRKstep}

\subsubsection{Notions of stability for Runge--Kutta methods}
We denote by $\C^-$ the set of complex numbers with non positive real part,
and by $i\R$ the set of purely imaginary numbers.
Following \cite{Crouzeix80}, \cite{Hairer82}, \cite{BHV86} and \cite{DL2023RK},
we recall the following definitions.

\begin{definition}
\label{def:stab}  
  A Runge--Kutta method is said to be
\begin{itemize} 
  	\item $A$-stable if for all $\lambda\in \C^-$, $|R(\lambda)|\leq 1$.
	\item $I$-stable if for all $\lambda\in i\R$, $|R(\lambda)|\leq 1$.
        \item \emph{$AS$-stable} if the rational function
          $\lambda\mapsto \lambda b^t(I-\lambda A)^{-1}$ has only removable singularities in $\C^-$
          and is bounded on $\C^-$.
	\item $ASI$-stable if for all $\lambda\in\C^-$, the matrix $(I-\lambda A)$ is
  invertible and $\lambda\mapsto (I-\lambda A)^{-1}$ is uniformly bounded on $\C^-$.
        \item \emph{$IS$-stable} if the rational function
          $\lambda\mapsto \lambda b^t(I-\lambda A)^{-1}$ has only removable singularities in $i\R$
          and is bounded on $i\R$.
	\item $ISI$-stable if for all $\lambda\in i\R$, the matrix $(I-\lambda A)$ is
  invertible and $\lambda\mapsto (I-\lambda A)^{-1}$ is uniformly bounded on $i\R$.
	\item $\hat{A}$-stable if it is $A$-stable, $AS$-stable and $ASI$-stable.
	\item $\hat{I}$-stable if it is $I$-stable, $IS$-stable and $ISI$-stable.
\end{itemize}
\end{definition}

A few examples of Runge--Kutta methods and their stability properties
can be found in Section \ref{subsec:examples}. The analysis of the stability properties of these methods as well as other examples can be found in \cite{DL2023RK}.

\begin{remark}
  \label{rem:AdonneI}
  Observe that, from this definition, an $A$-stable Runge--Kutta collocation method is also $I$-stable.
  Similarly, an $AS$-stable Runge--Kutta collocation method is also $IS$-stable, and
  and $ASI$-stable Runge--Kutta collocation method is also $ISI$-stable.
  As a consequence, an $\hat A$-stable method is also $\hat I$-stable.
\end{remark}

\begin{remark}
  \label{rem:SIimpliqueS}
  If $b$ is in the range of $A^t$, then $ASI$-stability implies $AS$-stability using Lemma
  4.4 of \cite{BHV86}. Similarly, $ISI$-stability implies $IS$-stability. 
\end{remark}

\begin{remark}
  Further connections between these eight notions of stability for collocation Runge--Kutta methods are analysed in \cite{DL2023RK}.
\end{remark}

\begin{remark}
The well posedness of the {\bf linear} PDE \eqref{eq:general} (with $N \equiv 0$) relies on the fact that the spectrum of the linear operator $L$ is included in a half-plane of the form 
$\{z\in\C \ |\ \Re(z) \leq s_0\}$ for some $s_0$ in $\R$. Up to a change of unknown, we may consider the case $s_0=0$ without loss of generality. The first example $L=\Delta$ (see Section \ref{sec:intro}) fits this framework.
The analysis of our methods applied to this example will involove the notion of $\hat{A}$-stability even in the nonlinear setting. Similarly, the second example $L=i\Delta$ (see Section \ref{sec:intro}) also fits this framework.
The analysis of our methods applied to this example will involove the notion of $\hat{I}$-stability even in the nonlinear setting.
\end{remark}

\subsubsection{{\color{black}Stability of Runge--Kutta methods for semilinear PDE problems}}

In this paper, we consider semilinear evolution problems of the form \eqref{eq:general},
with $L=\Delta$ or $L=i\Delta$ as explained in Section \ref{subsec:functionalframework}. 
What follows can be adapted for any diagonalizable operator on $H^\sigma$, provided one knows
the localization of the spectrum. The goal of this section is to set notations on the operators
that will be used later in the paper and to establish estimates on these operators,
provided that the Runge--Kutta method has some stability property.
The proof of these estimates is presented in Appendix (Section \ref{subsec:proofASIstable}).

\begin{proposition}
\label{prop:ASstable}
  Assume the Runge--Kutta method is $AS$-stable and $L=\Delta$.
  There exists a constant $C>0$ such that for all $h>0$, we have
  \begin{equation*}
    \|\left(h b^t \otimes L\right) \left(I-hA\otimes L\right)^{-1}\|_{(H^\sigma)^s\to H^\sigma}
    \leq C.
  \end{equation*}
\end{proposition}

\begin{proposition}
\label{prop:ISstable}
  Assume the Runge--Kutta method is $IS$-stable and $L=i\Delta$.
  There exists a constant $C>0$ such that for all $h>0$, we have
  \begin{equation*}
    \|\left(h b^t \otimes L\right) \left(I-hA\otimes L\right)^{-1}\|_{(H^\sigma)^s\to H^\sigma}
    \leq C.
  \end{equation*}
\end{proposition}

\begin{proposition}
  \label{prop:ASIstable}
  Assume the Runge--Kutta method is $ASI$-stable and $L=\Delta$.
  There exists a constant $C>0$ such that for all $h>0$, we have
  \begin{equation*}
    \|\left(I-hA\otimes L\right)^{-1}\|_{(H^\sigma)^s\to (H^\sigma)^s}
    \leq C.
  \end{equation*}
\end{proposition}

\begin{proposition}
  \label{prop:ISIstable}
  Assume the Runge--Kutta method is $ISI$-stable and $L=i\Delta$.
  There exists a constant $C>0$ such that for all $h>0$, we have
  \begin{equation*}
    \|\left(I-hA\otimes L\right)^{-1}\|_{(H^\sigma)^s\to (H^\sigma)^s}
    \leq C.
  \end{equation*}
\end{proposition}

\subsubsection{Well-posedness of the quasi-Runge--Kutta step}

\begin{proposition}\label{prop:RKwd}
  If the chosen Runge--Kutta method is ASI-stable and $L=\Delta$ or if it is ISI-stable and $L=i\Delta$,
  then the internal linear step \eqref{eq:systlinvect} of the linearly implicit method
  \eqref{eq:heritagegamma}-\eqref{eq:systlinvect}-\eqref{eq:RKexplvect} is well defined for $h>0$ sufficiently
  small, where ``sufficiently small'' depends only on the norm of $\Gamma_n$ in $(H^\sigma)^s$.
\end{proposition}

\begin{proof}
  Assume the method is ISI-stable and $L=i\Delta$ (the proof follows the very same lines in the other case).
  Since the method is ISI-sable and $L=i\Delta$, proposition \ref{prop:ISIstable} ensures
  that for all $h>0$, the operator $(I-hA\otimes L)$ has bounded inverse from $(H^\sigma)^s$ to itself,
  and its operator norm is bounded by a constant that does not depend on $h$.
  Solving \eqref{eq:systlinvect} in $(H^\sigma)^s$ is indeed equivalent to solving for $U_n$ in $(H^\sigma)^s$
  \begin{equation*}
    \left(I-h\left(I-hA\otimes L\right)^{-1}A\Gamma_n\bullet\right) U_n
    = \left(I - h A\otimes L\right)^{-1}u_n\mathds{1}.
  \end{equation*}
  Since the operator norm of $\left(I-hA\otimes L\right)^{-1}A\Gamma_n\bullet$ is bounded by the product
  of the operator norm of $\left(I-hA\otimes L\right)^{-1}$ by that of $U\mapsto A\Gamma_n\bullet U$,
  and since the latter only depends on a bound on $\Gamma_n$, the result follows using a Neumann series argument.
\end{proof}

\subsection{Hypotheses and main result}
\label{subsec:hypoandresult}
With the hypotheses on the exact solution and on $T$ introduced
in Section \ref{ssec:consistencyerrors},
we choose some $r>0$ and we denote by $V$ an $r$-neighbourhood of the exact solution
of \eqref{eq:general} defined by
\begin{equation}
  \label{eq:defV}
  V=\{u(t)+v,\ |\ t\in[0,T]\text{ and } v\in \Hs\text{ with }\|v\|_\Hs \leq r\}.
\end{equation}

We make the following hypotheses on the nonlinearity $N$ in \eqref{eq:general} and on the
exact solution of the Cauchy problem.
\begin{hypo}
  \label{hyp:Nbb}
  The function $N$ sends bounded sets of $\Hs$ to bounded sets of $\Hs$.
\end{hypo}

\begin{hypo}
  \label{hyp:NLip}
  The function $N$ is Lipschitz continuous on bounded sets of $\Hs$.
\end{hypo}

\begin{hypo}
  \label{hyp:uex}
  The exact solution $u$ has $s+2$ continuous derivatives with values in $\Hs$ on $(T_\star,T^\star)$.
\end{hypo}

\begin{hypo}
  \label{hyp:Nou}
  The function $t\mapsto N(u(t))$ has $s$ continuous derivatives in $H^\sigma$ on $(T_\star,T^\star)$.
\end{hypo}

We make the following assuptions on the numerical method
\eqref{eq:heritagegamma}-\eqref{eq:systlin}-\eqref{eq:RKexpl}.

\begin{hypo}
  \label{hyp:RKcol}
  The Runge--Kutta method is a collocation method with $s$ distinct points
  \break $0\leq c_1<\dots<c_s\leq 1$ (see Definition \ref{def:colmeth}).
\end{hypo}

\begin{hypo}
  \label{hyp:linimpl}
  The matrix $D\in\mathcal M_s(\C)$ and the vector $\Theta\in\C^s$ are such that
  the step \eqref{eq:heritagegamma} is strongly stable and consistent of order $s$.
  We denote by $\delta$ a fixed number in $(\rho(D),1)$, and by $\|\cdot\|_{\Hss,D}$
  the norm on $\Hss$ provided by Lemma \ref{lem:normeD}.
\end{hypo}

Using Hypothesis \ref{hyp:Nou}, we have that $N\circ u([0,T])$ is a bounded set of $\Hs$ and we denote
by $M>0$ a bound on this set. Moreover, with Hypothesis \ref{hyp:Nbb}, the set $N(V)$ is bounded
and we denote by $m>0$ a constant such that
\begin{equation}
  \label{eq:Mnm}
  \forall\, v\in V,\qquad \|N(v)\|_\Hs < M+m.
\end{equation}

Let us recall that we chose $\sigma>d/2$ so that $H^\sigma$ is an algebra
(see Section \ref{subsec:functionalframework}).

\begin{theorem}\label{th:convergence}
  Assume that the function $N$ in \eqref{eq:general} satisfies hypotheses \ref{hyp:Nbb} and
  \ref{hyp:NLip}.
  Assume $u^0\in\Hs$ is fixed and that $T\in(0,T^\star)$, $r>0$ and $V$
  are defined as in \eqref{eq:defV}.
  Assume the exact solution satisfies hypotheses \ref{hyp:uex} and \ref{hyp:Nou}.
  Assume $M>0$ and $m>0$ are defined as in \eqref{eq:Mnm}.

  Assume that the underlying Runge--Kutta method satisfies hypothesis \ref{hyp:RKcol}
  and that the step \eqref{eq:heritagegamma} satisfies hypothesis \ref{hyp:linimpl}.
  Assume that the underlying Runge--Kutta method is $\hat{I}$-stable and $L=i\Delta$,
  or that the underlying Runge--Kutta method is $\hat{A}$-stable and $L=\Delta$.

  There exists a constant $C>0$, a small $h_0\in (0,T)$ and a neighbourhood of $u^0\in \Hs$
  such that for all $u_0$ in this neighbourhood, all $h\in (0,h_0)$
  and all $\gamma_{-1,1},\dots,\gamma_{-1,s}$ sufficiently close their continuous
  analogues $N(u(t_{-1}+c_1h)),\dots,N(u(t_{-1}+c_sh))$,
  the numerical method \eqref{eq:heritagegamma}, \eqref{eq:systlin}, \eqref{eq:RKexpl} is well defined
  for all $n\in\N$ such that $0\leq nh\leq T$.
  Moreover, for such $n$ and $h$, the method satisfies
  \begin{equation}
    \label{eq:estimgamma}
    \forall i\in\{1,\dots,s\},\qquad \|\gamma_{n-1,i}\|_{\Hs}\leq M+m,
  \end{equation}
  and
  \begin{equation}
    \label{eq:estimerreur}
    \max_{0\leq k \leq n} \|u(t_k)-u_k\|_\Hs \leq e^{Cnh} \left(
      \left\|u^0-u_0\right\|_\Hs + C \left(\max_{i\in\{1,\dots, s\}} \left\|N(u(t_{-1}+c_ih))-\gamma_{-1,i}\right\|_\Hs+h^s
      \right)
    \right).
  \end{equation}
\end{theorem}

\subsection{Proof of Theorem \ref{th:convergence}}
\label{subsec:proof}

Let us first introduce all the notations we use in the proof of the Theorem. Let us define the convergence errors $P_n \in \Hss$ with component number $i$ equal to $(P_n)_i=N(u(t_n+c_ih))-\gamma_{n,i}$, $Q_n \in \Hss$ with component number $i$ equal to $Q_{n,i}=u(t_n+c_ih)-u_{n,i}$ (provided $u_{n,i}$ is well defined), and $e_n \in \Hs$ with $e_n=u(t_n)-u_n$. We set $z_n=\max_{0 \leq k \leq n} \|e_k\|_{\Hs}$.  
In the following proof, the letter $C$ denotes a real number greater or equal to $1$,
which does not depend on $h$ (but depends on $M$ and $r$ in particular)
and whose value may vary from one line to the other.  

\vskip0.3cm
Using Hypothesis \ref{hyp:linimpl}, step \eqref{eq:heritagegamma} is strongly stable.
Therefore, we fix $\delta \in (\rho(D),1)$ and we use Lemma \ref{lem:normeD}
to define a norm $\|\cdot\|_{\Hss,D}$ such that \eqref{eq:inegnormeD} holds. 


We divide the proof in two parts. First, we assume an a priori bound for the numerical solution.
Assume we are given an integer $\nu\in\N$ such that $t_{\nu+1}\leq T$ and for all $n\leq \nu$,
\begin{itemize}
	\item (H1) $\|\Gamma_n\|_{\Hss,\infty} \leq M+m$,
	\item (H2) the step \eqref{eq:systlin} has a unique solution $(u_{n,i})_{1\leq i \leq s}$ in $\Hss$,  
	\item (H3) $u_n \in V$. 
\end{itemize}
We show that, in this case, we have an explicit bound for the convergence errors $(P_n)_{0\leq n\leq \nu}$ and $(z_n)_{0\leq n\leq \nu}$ (see equations \eqref{eq:recestznfin} and \eqref{eq:Pnderoule3}). 

Second, we assume that $h_0$ and the initial errors $P_{-1}$ and $e_0$ are small enough and we show that the bounds of the first part of the proof are indeed satisfied. 

\vskip0.2cm
\noindent{\bf First part.} In addition to the bounds above, we assume that $h\in (0,1)$.
Let us consider $n\in\N$ with $n\leq \nu$. In particular, we have $t_{n+1} \leq T$.  
Substracting \eqref{eq:heritagegamma} from \eqref{eq:Rn1} we obtain
\begin{equation}
\label{eq:recerrPn}
P_n=DP_{n-1}+ \left(N(u(t_n))-N(u_n)\right)\left [
       \begin{matrix}
       \theta_1\\
       \vdots\\
       \theta_s
       \end{matrix}
     \right ]
+ R_n^1.
\end{equation}
We infer using estimate \eqref{eq:inegnormeD} of Lemma \ref{lem:normeD}
\begin{equation}
\label{eq:recestPn}
\|P_n\|_{\Hss,D} \leq \delta \|P_{n-1}\|_{\Hss,D} + C \|e_n\|_{\Hs} + \|R_n^1\|_{\Hss,D},
\end{equation}
where the constant $C$ is proportional to the Lipschitz constant of $N$ over the bounded set $V$ (using hypothesis \ref{hyp:NLip}). 

Substracting \eqref{eq:systlin} from \eqref{eq:Rn2} we obtain
\begin{eqnarray*}
\lefteqn{Q_{n,i}}\\
&=&e_n + h \sum_{j=1}^s a_{i,j} \left(LQ_{n,j}+N(u(t_n+c_jh))u(t_n+c_jh)-\gamma_{n,j}u_{n,j}\right) +(R_n^2)_i \\
&=& e_n + h \sum_{j=1}^s a_{i,j} LQ_{n,j}+ h \sum_{j=1}^s a_{i,j} \left(N(u(t_n+c_jh))-\gamma_{n,j}\right)u(t_n+c_jh) \\
&{}& \hskip3.5cm+h \sum_{j=1}^s a_{i,j} \gamma_{n,j}\left(u(t_n+c_jh)-u_{n,j}\right) +(R_n^2)_i \\
  &=& e_n + h \sum_{j=1}^s a_{i,j} LQ_{n,j}
      + h \sum_{j=1}^s a_{i,j} P_{n,j}u(t_n+c_jh)
      +h \sum_{j=1}^s a_{i,j} \gamma_{n,j}Q_{n,j}
      +(R_n^2)_i.
\end{eqnarray*}
Therefore, the vector $Q_n$ solves 
\begin{equation}
\label{eq:systQn}
(I-hA\otimes L)Q_n= e_n\mathds{1} +hA P_n \bullet U_{t_n}+ hA\Gamma_n \bullet Q_n+R_n^2,
\end{equation}
where $U_{t_n}$ is the vector with component $j$ equal to $u(t_n+c_jh)$.
Since the Runge-Kutta method is either $\hat{I}$-stable or $\hat{A}$-stable,
it is either ISI-stable or ASI-stable and the operator $I-hA\otimes L$ is invertible
(see propositions \ref{prop:ASIstable} and \ref{prop:ISIstable})
and the equation above is equivalent to 
$$(I-h(I-hA\otimes L)^{-1}A\Gamma_n \bullet)Q_n= (I-hA\otimes L)^{-1}\left(e_n\mathds{1} +hA P_n \bullet U_{t_n}+R_n^2\right).$$
Just as in proposition \ref{prop:RKwd} the operator in the left hand side is invertible for $h$ small depending only on a bound on $\|\Gamma_n\|_{\Hss,\infty}$. Let us denote by $\| \cdot \|_{\star}$ the operator norm in $(\Hss,\| \cdot \|_{\Hss, \infty})$. As soon as $h\| (I-hA\otimes L)^{-1}A\Gamma_n \bullet\|_{\star}<1$, we have
$$ \|Q_n\|_{\Hss,\infty} \leq C \dfrac{1}{1-h\| (I-hA\otimes L)^{-1}A\Gamma_n \bullet\|_{\star}} \left(\|e_n\|_\Hs + Ch \|P_n\|_{\Hss,\infty} + \|R_n^2\|_{\Hss,\infty}\right).$$
Note that, using (H1), $\|A\Gamma_n \bullet\|_{\star} \leq C_1 (M+m)$, where $C_1$ only depends on the matrix $A$. Since $\| (I-hA\otimes L)^{-1}\|_{\Hss \to \Hss}$ is bounded by a constant that does not depend on $h$, the same holds for $\| (I-hA\otimes L)^{-1}\|_{\star}$ with a constant $C_2$. Then, provided that $h$ is sufficiently small to ensure that $h C_1C_2(M+m) \leq 1/2$, we have
\begin{equation}
\label{eq:recestQn}
\|Q_n\|_{\Hss,\infty} \leq C\|e_n\|_\Hs + Ch \|P_n\|_{\Hss,\infty}+ C\|R_n^2\|_{\Hss,\infty}. 
\end{equation}

Substracting \eqref{eq:RKexpl} from \eqref{eq:Rn3} we obtain 
\begin{eqnarray}
e_{n+1}&=& e_n + h \sum_{i=1}^s b_i L Q_{n,i}+ h \sum_{i=1}^s b_i \left(N(u(t_n+c_ih))-\gamma_{n,i}\right)u(t_n+c_ih) \nonumber\\
&{}&\hskip3.5cm +h \sum_{i=1}^s b_i \gamma_{n,i}\left(u(t_n+c_ih)-u_{n,i}\right) +R_n^3  \nonumber\\
&=& e_n + h \sum_{i=1}^s b_i L Q_{n,i}+ h \sum_{i=1}^s b_i P_{n,i}u(t_n+c_ih) +h \sum_{i=1}^s b_i \gamma_{n,i}Q_{n,i} +R_n^3 \nonumber\\
&=& e_n + h (b^t \otimes L )Q_n + h b^t (P_n \bullet U_{t_n}) + h  b^t (\Gamma_n \bullet Q_n) + R_n^3 \nonumber\\
&=& e_n + h( b^t \otimes L )(I-hA\otimes L)^{-1}e_n\mathds{1}+h (b^t \otimes L) (I-hA\otimes L)^{-1}\left[hA P_n \bullet U_{t_n}+ hA\Gamma_n \bullet Q_n+R_n^2\right] \nonumber \\
&{}& \hskip5.3cm+ h b^t P_n \bullet U_{t_n} + h  b^t \Gamma_n \bullet Q_n + R_n^3, \label{eq:topourri}
\end{eqnarray}
where we have used \eqref{eq:systQn}.

Let us denote by $v_n= e_n + h( b^t \otimes L )(I-hA\otimes L)^{-1}e_n\mathds{1}$.
Assume that the spatial domain is $\Omega=\R^d$.
Observe that if $L=i\Delta$, then for all $\xi\in\R^d$,
$$ \widehat{v_n}(\xi) = R(-ih\xi^2)\widehat{e_n}(\xi),$$
and that if $L=\Delta$, then for all $\xi\in\R^d$,
$$ \widehat{v_n}(\xi) = R(-h\xi^2)\widehat{e_n}(\xi),$$
where $R$ is the linear stability function of the Runge-Kutta method defined in \eqref{eq:defR}.
Since the Runge-Kutta method is either $\hat{I}$-stable or $\hat{A}$-stable, it is either I-stable or A-stable and in both cases, hence $|R|\leq1$ in the appropriate region.
Similar estimates hold in the periodic case $\Omega=\T^d$.
Therefore,
$$\| v_n \|_{\Hs} \leq \| e_n \|_{\Hs}.$$
Since the Runge-Kutta method is either $\hat{I}$-stable or $\hat{A}$-stable, it is either IS-stable or AS-stable and in both cases, using \eqref{eq:topourri}, we infer with
propositions \ref{prop:ASstable} and \ref{prop:ISstable}, that
\begin{equation*}
\|e_{n+1}\|_{\Hs} \leq \|e_n\|_\Hs + Ch \|P_n\|_{\Hss,\infty} + Ch \|\Gamma_n\|_{\Hss,\infty} \|Q_n\|_{\Hss,\infty}+ C\|R_n^2\|_\Hss + \|R_n^3\|_\Hs , 
\end{equation*}
which gives with (H1)
\begin{equation*}
\|e_{n+1}\|_\Hs \leq \|e_n\|_\Hs + Ch \|P_n\|_{\Hss,\infty} + Ch(M+m) \|Q_n\|_{\Hss,\infty} + C\|R_n^2\|_\Hss + \|R_n^3\|_\Hs. 
\end{equation*}
Using \eqref{eq:recestQn}, and recalling that $h\in(0,1)$ so that $h^2\leq h\leq 1$, we have
\begin{equation}
\label{eq:recesten}
\|e_{n+1}\|_\Hs \leq (1+Ch)\|e_n\|_\Hs + Ch \|P_n\|_{\Hss,\infty}+ C \|R_n^2\|_\Hss+ \|R_n^3\|_\Hs. 
\end{equation}
From \eqref{eq:recestPn} we have by induction 
\begin{equation}
\label{eq:Pnderoule}
\|P_n\|_{\Hss,D }\leq \delta^{n+1} \|P_{-1}\|_{\Hss,D} + C \sum_{k=0}^n \delta^{n-k} \left(\|e_k\|_\Hs+ \|R_k^1\|_{\Hss,D}\right).
\end{equation}
Using the norm equivalence between $\| \cdot \|_{\Hss,\infty}$ and $\| \cdot \|_{\Hss,D}$, and \eqref{eq:Pnderoule} in \eqref{eq:recesten} we obtain
\begin{eqnarray}
\|e_{n+1}\|_\Hs &\leq& (1+Ch)\|e_n\|_\Hs + Ch \left[\delta^{n+1} \|P_{-1}\|_{\Hss,D} + C \sum_{k=0}^n \delta^{n-k} \left(\|e_k\|_\Hs+ \|R_k^1\|_{\Hss,D}\right)\right] \nonumber \\
&{}& \hskip3.5cm+ C\|R_n^2\|_\Hss+ \|R_n^3\|_\Hs. \label{eq:recesten2} 
\end{eqnarray}
With estimate \eqref{eq:estimRn23}, we infer that 
\begin{equation}
\label{eq:recesten3}
\|e_{n+1}\|_\Hs \leq (1+Ch)\|e_n\|_\Hs + Ch \left[\delta^{n+1} \|P_{-1}\|_{\Hss,D} + C \sum_{k=0}^n \delta^{n-k} \left(\|e_k\|_\Hs+ \|R_k^1\|_{\Hss,D}\right)\right]+ Ch^{s+1}.
\end{equation}
Using the maximal error defined previously and the fact that $\delta <1$ since the step \eqref{eq:heritagegamma} is strongly stable, we have, using also estimate \eqref{eq:estimRn1},
\begin{eqnarray*}
\|e_{n+1}\|_\Hs &\leq& (1+Ch)z_n + Ch \left[\delta^{n+1} \|P_{-1}\|_{\Hss,D} + \left(z_n+ h^s\right)\sum_{k=0}^n \delta^{n-k} \right]+ Ch^{s+1} \\
&\leq& (1+Ch)z_n + Ch \left[\delta^{n+1} \|P_{-1}\|_{\Hss,D} + \left(z_n+ h^s\right)\dfrac{1}{1-\delta} \right]+ Ch^{s+1} ,
\end{eqnarray*}
and then 
\begin{equation}
\label{eq:recesten4}
\|e_{n+1}\|_\Hs \leq (1+Ch)z_n + Ch \delta^{n+1} \|P_{-1}\|_{\Hss,D} + Ch^{s+1}.
\end{equation}
Using that $z_{n+1}=\max \{z_n, \|e_{n+1}\|_\Hs\}$, we infer
\begin{equation}
\label{eq:recesten5}
z_{n+1}\leq (1+Ch)z_n + Ch  \|P_{-1}\|_{\Hss,D} + Ch^{s+1}.
\end{equation}
By induction it follows that for all $n$ in $\N$ such that $t_{n}\leq T$,
\begin{eqnarray}
z_{n}&\leq& (1+Ch)^n z_0 + Ch  \left(\|P_{-1}\|_{\Hss,D} + h^{s}\right)\sum_{k=0}^{n-1} (1+Ch)^k  \nonumber \\
&\leq& e^{Cnh} z_0 + Ch  \left(\|P_{-1}\|_{\Hss,D} + h^{s}\right)\dfrac{(1+Ch)^{n}}{1+Ch-1} \nonumber \\
&\leq& e^{Cnh}\left( z_0 + C \left(\|P_{-1}\|_{\Hss,D} +h^{s}\right)\right)\nonumber \\
&\leq& e^{Cnh}\left( z_0 + C \left(\|P_{-1}\|_{\Hss,\infty} +h^{s}\right)\right). \label{eq:recestznfin}
\end{eqnarray}
Using \eqref{eq:Pnderoule} and the same estimations as above, we have moreover, since $\delta<1$,
\begin{equation}
\label{eq:Pnderoule2}
\|P_n\|_{\Hss,D} \leq \|P_{-1}\|_{\Hss,D} + C(z_n+h^s).
\end{equation}
We infer, using \eqref{eq:recestznfin},
\begin{equation}
\label{eq:Pnderoule3}
\|P_n\|_{\Hss,\infty} \leq Ce^{Cnh}\left(z_0+\|P_{-1}\|_{\Hss,\infty} +h^s\right).
\end{equation}

This shows that, as long as $(H1)$, $(H2)$ and $(H3)$ hold for $n$ such that $0\leq n\leq \nu$,
the estimates \eqref{eq:recestznfin} and \eqref{eq:Pnderoule3}) hold for
the convergence errors $(P_n)_{0\leq n\leq \nu}$ and $(z_n)_{0\leq n\leq \nu}$.
Observe that the constant $C>1$ above does not depend on $h$, neither does it depend
on $\nu$. However, it depends on the exact solution through the hypotheses stated before the theorem.
This concludes the first part of the proof.
\vskip0.2cm
\noindent{\bf Second part.} 
From now on, we denote by $C$ the maximum of the constants appearing in the right hand sides of \eqref{eq:recestznfin} and \eqref{eq:Pnderoule3}. Choose $h_0 \in (0,1)$ sufficiently small to have $h_0<\min\{-T_\star,T^\star\}$ and $Ce^{CT}h_0^s<r$ and $Ce^{CT}h_0^s<m$ and $h_0 C_1C_2(M+m) \leq 1/2$ where $C_1$ and $C_2$ were defined in the first part. Assume $u_0, \gamma_{-1+c_1}, \dots, \gamma_{-1+c_s} \in \Hs$ and $h \in (0,h_0)$ satisfy 
\begin{equation}
\label{eq:contrainteu0}
e^{CT}\left( \|u^0-u_0\|_{\Hs} + C \left(\max_{i \in \ldbrack 1,s \rdbrack} \|\gamma_{-1+c_i} - N(u(t_{-1}+c_ih))\|_{\Hs} +h_0^{s}\right)\right)<r,
\end{equation} 
and
\begin{equation}
\label{eq:contraintegammainit}
Ce^{CT}\left(\|u^0-u_0\|_{\Hs}+\max_{i \in \ldbrack 1,s \rdbrack} \|\gamma_{-1+c_i} - N(u(t_{-1}+c_ih))\|_{\Hs} +h_0^s\right)<m.
\end{equation} 

First, with \eqref{eq:Pnderoule3} (for $n=0$) and \eqref{eq:contraintegammainit}, we have 
$$\|P_{0}\|_{\Hss, \infty} = \max_{i \in \ldbrack 1,s \rdbrack} \|\gamma_{0+c_i} - N(u(t_{0}+c_ih))\|_{\Hs} < m.$$
Therefore by triangle inequality we have 
$$\|\Gamma_{0}\|_{\Hss,\infty} \leq \|P_{0}\|_{\Hss,\infty} + \left\| (N(u(t_{0}+c_ih)))_{1\leq i \leq s} \right\|_{\Hss,\infty}<M+m.$$
And then, the hypothesis (H1) of the first part is satisfied for $n=0$.

Moreover, with \eqref{eq:contrainteu0}, we have $\|u^0-u_0\|_\Hs \leq r$ so that $u_0 \in V$ and the hypothesis (H3) of the first part is satisfied with $n=0$. With proposition \ref{prop:RKwd}, we infer that the system \eqref{eq:systlinvect} has a unique solution in $\Hss$ since we assumed $h \leq h_0 < 1/(2C_1C_2(M+m))$. This implies that the hypothesis (H2) of the first part is satisfied for $n=0$. Then we can apply the analysis of the first part (with $\nu=0$) to obtain \eqref{eq:recestznfin} and \eqref{eq:Pnderoule3} with $n=1$. Using \eqref{eq:contrainteu0} and \eqref{eq:contraintegammainit}, we infer that hypotheses (H1), (H2) and (H3) are satisfied for all $n\leq \nu=1$, and the result follows by induction on $\nu$.



\section{Remark on the mass conservation for NLS equation}
\label{sec:massconservation}

Let us consider the NLS equation \eqref{eq:general} with $L=i\Delta$ and $N$ taking values in $i\R$, such that the Cauchy problem associated to \eqref{eq:general} is well-posed. In this case, the $L^2$-norm of the solution is preserved by the exact flow (at least in the case $\Omega=\R^d$ and $\Omega=\T^d$). Indeed, classically, one has along the exact solution $u$:
\begin{equation*}
\frac12\frac{\rm d}{{\rm d}t}\|u(t)\|^2=\Re\left(\int_\Omega \left(-i|\nabla u|^2 + N(u)|u|^2\right){\rm d}x {\rm d}y\right)=0.
\end{equation*}
For linearly implicit methods, it is possible to reproduce this property numerically, by adapting a classical condition for the preservation of quadratic invariants by Runge--Kutta methods
\cite{GNI2002}.

\begin{proposition}
\label{prop:Coopercondition}
  Assume $L=i\Delta$ and $N$ takes values in $i\R$. Assume the collocation Runge--Kutta method
  satisfies the Cooper condition
  (see equation (3.9) in Remark 14 of \cite{DL2020} or \eqref{eq:coopercondition} below).
  Assume the eigenvalues of $D$ are chosen so that $D$ is a real valued squared matrix
  and $\theta$ is a real valued vector.
  If $\gamma_{-1,1}, \cdots, \gamma_{-1,s}$ take values in $i\R$,
  then for all $u_0 \in L^2(\Omega)$ and for all $n \in \N$ such that
  $u_n$ is well-defined by the linearly implicit method, one has
\begin{equation*}
\|u_n\|^2=\|u_0\|^2.
\end{equation*} 
\end{proposition}

\begin{proof}
First of all, let us recall de Cooper condition for a Runge--Kutta collocation method:
\begin{equation}
\label{eq:coopercondition}
\forall 1 \leq i,j \leq s, \qquad b_ib_j-b_ia_{ij}-b_ja_{ji}=0.
\end{equation}
Let $n$ be given as in the hypothesis. We have, using \eqref{eq:RKexpl}:
\begin{eqnarray}
\int_\Omega |u_{n+1}|^2&=&\int_\Omega u_{n+1} \overline{u_{n+1}} \nonumber\\ 
&=&\int_\Omega \left[\left(u_n+h\sum_{i=1}^s b_i(L+\gamma_{n+c_i})u_{n,i}\right)\overline{\left(u_n+h\sum_{i=1}^s b_i(L+\gamma_{n+c_i})u_{n,i}\right)}\right]\nonumber\\
&=&\int_\Omega |u_n|^2+h\int_\Omega I_1+h \int_\Omega I_2+h^2\int_\Omega I_3, \label{eq:calculmasse}
\end{eqnarray}
where, since $b_i\in\R$,
\begin{eqnarray*}
I_1&=&\sum_{i=1}^s b_i \overline{(L+\gamma_{n+c_i})u_{n,i}}~u_n, \\
I_2&=&\sum_{i=1}^s b_i (L+\gamma_{n+c_i})u_{n,i}\overline{u_n}, \\
I_3&=&\sum_{i,j=1}^s b_ib_j (L+\gamma_{n+c_i})u_{n,i} 
\overline{(L+\gamma_{n+c_j})u_{n,j}}.
\end{eqnarray*}
Using \eqref{eq:systlin}, we can write:
\begin{eqnarray*}
I_1&=&\sum_{i=1}^s b_i \left(\overline{(L+\gamma_{n+c_i}})\overline{u_{n,i}}\right)u_{n,i}-h\sum_{i=1}^s b_i \left(\overline{(L+\gamma_{n+c_i}})\overline{u_{n,i}}\right)\left( \sum_{j=1}^s a_{i,j}(L+\gamma_{n+c_j})u_{n,j}\right),\\
&=&\sum_{i=1}^s b_i \left(\overline{(L+\gamma_{n+c_i}})\overline{u_{n,i}}\right)u_{n,i}-h\sum_{j=1}^s b_j \left(\overline{(L+\gamma_{n+c_j}})\overline{u_{n,j}}\right)\left( \sum_{i=1}^s a_{j,i}(L+\gamma_{n+c_i})u_{n,i}\right),\\
&=&\sum_{i=1}^s b_i \left(\overline{(L+\gamma_{n+c_i}})\overline{u_{n,i}}\right)u_{n,i}-h\sum_{j=1}^s \left(\sum_{i=1}^s b_ja_{j,i}(L+\gamma_{n+c_i})u_{n,i}\overline{(L+\gamma_{n+c_j}})\overline{u_{n,j}}\right).
\end{eqnarray*}
A similar computation leads to
\begin{eqnarray*}
I_2&=&\sum_{i=1}^s b_i \left((L+\gamma_{n+c_i})u_{n,i}\right)\overline{u_{n,i}}-h\sum_{j=1}^s \left(\sum_{i=1}^s b_ia_{i,j}(L+\gamma_{n+c_i})u_{n,i}\overline{(L+\gamma_{n+c_j}})\overline{u_{n,j}}\right).
\end{eqnarray*}
Using these two equalities in \eqref{eq:calculmasse}, we obtain
\begin{eqnarray*}
\int_\Omega |u_{n+1}|^2&=&\int_\Omega |u_n|^2 +h \sum_{i=1}^s b_i \int_\Omega \left[\overline{(L+\gamma_{n+c_i})u_{n,i}}~u_{n,i}+(L+\gamma_{n+c_i})u_{n,i}\overline{u_{n,i}}\right]\\
&{}&+h^2 \sum_{j=1}^s \sum_{i=1}^s \left(b_ib_j-b_ia_{i,j}-b_ja_{j,i}\right) \int_\Omega (L+\gamma_{n+c_i})u_{n,i}\overline{(L+\gamma_{n+c_j})u_{n,j}}.
\end{eqnarray*}
With the Cooper condition \eqref{eq:coopercondition} in the last term, this gives
\begin{equation}
\label{eq:calculmasse2}
\int_\Omega |u_{n+1}|^2=\int_\Omega |u_n|^2 +2h \sum_{i=1}^s b_i {\rm Re}\left(\int_\Omega (L+\gamma_{n+c_i})u_{n,i}\overline{u_{n,i}}\right).
\end{equation}
Observe that, since the $(\gamma_{-1+c_i})_{1\leq i\leq s}$ are purely imaginary-valued initially,
$D$ is a real-valued matrix, and $N$ takes values in $i\R$, the step \eqref{eq:heritagegamma}
ensures by induction that $(\gamma_{n+c_i})_{1\leq i\leq s}$ are purely imaginary-valued.
Using that and the fact that, for the Schrödinger equation, the operator $L=i\Delta$ is skew-symmetric,
we infer that the last term in \eqref{eq:calculmasse2} is equal to zero. This implies
$$\int_\Omega |u_{n+1}|^2=\int_\Omega |u_n|^2.$$
This concludes the proof. 
\end{proof}


\section{Numerical experiments}
\label{sec:num}

This section is devoted to numerical experiments illustrating the main theorem of the paper
(Theorem \ref{th:convergence}).
We also demonstrate numerically the necessity of the stability conditions of Section
\ref{subsec:analysisofthequasiRKstep} for the convergence result to hold.
Moreover, we compare the precision and efficiency of the linearly implicit methods
analysed in this paper with that of classical methods from the litterature.
We consider NLS equations in 1D (Section \ref{subsec:numNLS1D}) and 2D (Section \ref{subsec:numNLS2D}),
as well as NLH equations in 1D (Section \ref{subsec:numNLH1D}). A more extensive numerical comparison of linearly implicit methods for several semilinear evolution PDEs will be the object of a forthcoming paper.

\subsection{Note on the initialization of the linearly implicit methods}
\label{subsec:initialization}

Initializing a linearly implicit method \eqref{eq:heritagegamma}, \eqref{eq:systlin}, \eqref{eq:RKexpl} requires not only an approximation $u_0$ of $u^0$ but also $s$ approximations $\gamma_{-1,1}, \cdots, \gamma_{-1,s}$ of $N(u(t_{-1+c_1})), \cdots, N(u(t_{-1+c_s}))$. To do so, in the numerical experiments below, we use appropriate methods that ensure that the corresponding terms in the right hand-side of \eqref{eq:estimerreur} has order $s$. This is easy to do for NLS equations since they are reversible with respect to time. For non reversible equations like NLH equations, we may compute forward approximations of $N(u(t_{0+c_1})), \cdots, N(u(t_{0+c_s}))$ and $u(t_1)$, by standard methods of sufficient orders and use these values as initial data for the linearly implicit method after a shift of $h$ in time.

\subsection{Numerical experiments in 1 dimension for the NLS equation}
\label{subsec:numNLS1D}

We perform numerical experiments in dimension 1 ($\Omega=\R$), on the soliton solution
\begin{equation}
  \label{eq:solitonNLS}
  u(t,x)=\sqrt{\frac{2\alpha}{q}} {\rm sech}\left(\sqrt{\alpha}(x-x_0-c t)\right) {\rm e}^{i(\alpha+\frac{c^2}{4})t} {\rm e}^{i c (x-x_0-c t)/2)},
\end{equation}
to the NLS equation \eqref{eq:general} with $L=i\partial_x^2$ and $N(u)=iq|u|^2$ for some $q \in \R\setminus \{0\}$, $x_0$ real, $c$ real, $\alpha>0$.

We initialize the methods by a projection of this exact solution on the space grid for $u^0$ and by the image by $N$ of the projection of this exact solution on the space grid for $\gamma_{-1,1}, \cdots, \gamma_{-1,s}$. The final error is computed comparing to the projection on the space grid of the exact solution at final time. 

\subsubsection{Methods of order 2} \label{subsubsec:methodesordre21D}
The linearly implicit method with the underlying Runge--Kutta method of Example \ref{ex:meth2etagesLobatto} and $\pm \frac12$ as eignevalues of $D$,
applied to a nonlinear Schr\"odin\-ger equation with exact solution
\eqref{eq:solitonNLS} with $q=4$, $\alpha=1$,
$c=0$ and $x_0=0$, has been implemented in \cite{DL2020}. The computational domain was truncated to $(-50,50)$ with homogeneous
Dirichlet boundary conditions and discretized with $2^{18}$ points in space.
The final computational time was $T=5$. Observe that this method is $\hat A$-stable (see Example \ref{ex:meth2etagesLobatto}).

We refer the reader to Figure 4 in \cite{DL2020} for the numerical comparison with other convergent methods of order 2 from the liltterature
(the Crank-Nicolson method and the Strang splitting method). The left panel shows the order 2 of each method and in particular illustrates Theorem \ref{th:convergence} for the linearly implicit method. Moreover,  the right panel shows the efficiency of the methods in terms
of CPU time as a function of the error: The numerical cost for a given error is a bit higher for the linearly implicit method of order 2 than for the Strang splitting method and it is much lower than for the Crank--Nicolson method. 

To complete the illustration of Theorem \ref{th:convergence} for methods of order 2,
we investigate the relevance of the hypothesis of $\hat{A}$-stability for the underlying Runge--Kutta method.
To this end, we implement the linearly implicit method relying on the Runge--Kutta method described in Example \ref{ex:notAstable} with $\pm \frac12$ as eigenvalues of $D$
on the same soliton test case.
This underlying Runge--Kutta method is not $A$-stable (and even not $I$-stable) and hence not $\hat{A}$-stable.
We emphasize the fact that this linearly implicit method is of order 2
when applied to an ODE thanks to Theorem 9 in \cite{DL2020}.
However, Figure \ref{fig:ordre2_NLS1D_pasAstable} shows that this method fails to converge
for the soliton case.
This demonstrates the relevance of the hypothesis of $A$-stability (or at least $I$-stability)
for the numerical solution of nonlinear Schr\"odinger equations using linearly
implicit methods (this is in fact similar to the use of classical Runge--Kutta methods,
where $A$-stability is required as well to ensure convergence).

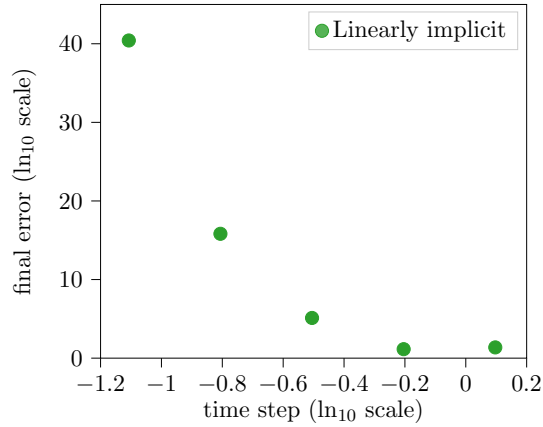
\begin{figure}[!h]
\centering
\resizebox{0.48\textwidth}{!}{
%
%
\definecolor{mycolor1}{rgb}{0.00000,0.44700,0.74100}%
\definecolor{color1}{rgb}{1,0.498039215686275,0.0549019607843137}
\definecolor{color0}{rgb}{0.172549019607843,0.627450980392157,0.172549019607843}
\begin{tikzpicture}

\begin{axis}[%
legend style={fill opacity=0.8, draw opacity=1, text opacity=1, draw=white!80!black},
tick align=outside,
tick pos=left,
xmin=-1.2,
xmax=0.2,
xlabel={time step ($\ln_{10}$ scale)},
xtick style={color=black},
x grid style={white!69.0196078431373!black},
ymin=0,
ymax=45,
y grid style={white!69.0196078431373!black},
ytick style={color=black},
ylabel={final error ($\ln_{10}$ scale)},
axis background/.style={fill=white},
]
\addplot [color=color0, draw=none, mark=*, mark size=3, mark options={solid}, only marks]
  table[row sep=crcr]{%
0.0969100130080564	1.36688700917022\\
-0.204119982655925	1.14218479154859\\
-0.505149978319906	5.11523985017124\\
-0.806179973983887	15.8212845020538\\
-1.10720996964787	40.4170392109337\\
};
\addlegendentry{Linearly implicit}

\end{axis}
\end{tikzpicture}%
}
\caption{Numerical error as a function of the time step for the soliton test case with the linearly implicit method based on the Runge--Kutta collocation method of Example \ref{ex:notAstable} which is not $\hat{A}$-stable (since not $A$-stable) with $\lambda_1=1/2, \lambda_2=-1/2$.}
\label{fig:ordre2_NLS1D_pasAstable}
\end{figure}

\subsubsection{Methods of order 4}
\label{subsubsec:methodesordre41D}

For the numerical methods of order 4, we use the parameters $q=8$, $x_0=0$, $c=0.5$, $\alpha=4$
in the nonlinear Schr\"odinger equation and its exact solution \eqref{eq:solitonNLS} above.
The computational domain is truncated to $(-62.5;62.5)$ and we set homogeneous Dirichlet boundary conditions. We use finite differences with $2^{20}$ points in space for each method. We denote by $\mathcal{L}$ the classical homogeneous Dirichlet Laplacian matrix times the purely imaginary unit. We perform the simulation until $T=5$.

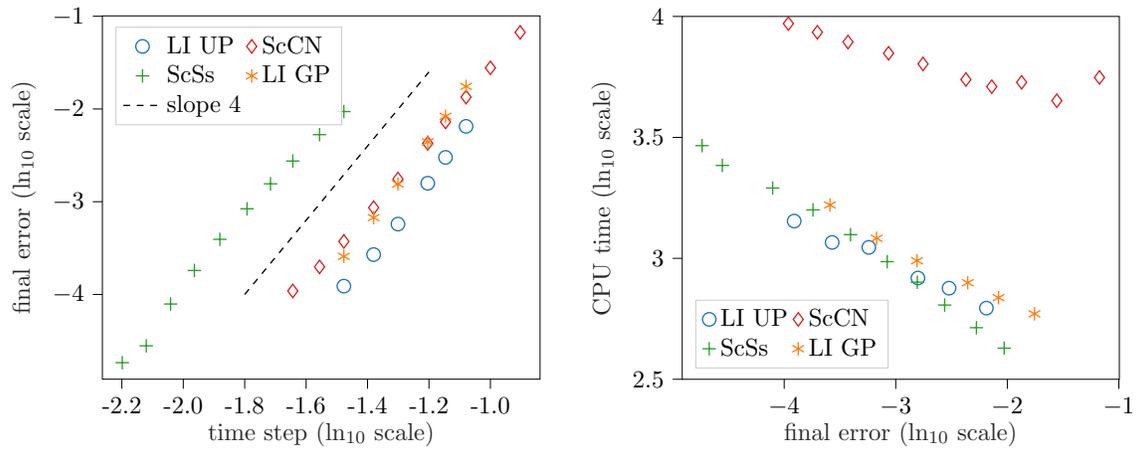
\begin{figure}[!h]
\centering
\begin{tabular}{cc}
\resizebox{0.48\textwidth}{!}{
\begin{tikzpicture}

\definecolor{color0}{rgb}{0.12156862745098,0.466666666666667,0.705882352941177}
\definecolor{color1}{rgb}{1,0.498039215686275,0.0549019607843137}
\definecolor{color2}{rgb}{0.172549019607843,0.627450980392157,0.172549019607843}
\definecolor{color3}{rgb}{0.83921568627451,0.152941176470588,0.156862745098039}
\definecolor{color4}{rgb}{0.580392156862745,0.403921568627451,0.741176470588235}
\definecolor{color5}{rgb}{0.549019607843137,0.337254901960784,0.294117647058824}

\begin{axis}[
legend cell align={left},
legend columns=2,
legend style={
  fill opacity=0.8,
  draw opacity=1,
  text opacity=1,
  at={(0.03,0.97)},
  anchor=north west,
  draw=white!80!black
},
tick align=outside,
tick pos=left,
x grid style={white!69.0196078431373!black},
xmin=-2.26343544195255, xmax=-0.83831163199382,
xlabel style={font=\color{white!15!black}},
xlabel={time step ($\ln_{10}$ scale)},
xtick style={color=black},
xtick={-2.4,-2.2,-2,-1.8,-1.6,-1.4,-1.2,-1,-0.8},
xticklabels={-2.4,-2.2,-2.0,-1.8,-1.6,-1.4,-1.2,-1.0,-0.8},
y grid style={white!69.0196078431373!black},
ymin=-4.9139019215706, ymax=-0.99587929877438,
ylabel style={font=\color{white!15!black}},
ylabel={final error ($\ln_{10}$ scale)},
ytick style={color=black}
]
\addplot [semithick, color0, mark=o, mark size=3, mark options={solid}, only marks]
table {%
-1.07918124604762 -2.18741930066588
-1.14612803567824 -2.52223857704097
-1.20411998265592 -2.80100179070028
-1.30102999566398 -3.2413790309208
-1.38021124171161 -3.57008569088546
-1.47712125471966 -3.91049376851996
};
\addlegendentry{LI UP}
\addplot [semithick, color3, mark=diamond, mark size=3, mark options={solid}, only marks]
table {%
-0.903089986991944 -1.17397123617421
-1 -1.55633961886013
-1.07918124604762 -1.87183394397232
-1.14612803567824 -2.13934516023454
-1.20411998265592 -2.37088736545236
-1.30102999566398 -2.75540685098658
-1.38021124171161 -3.06432268093115
-1.47712125471966 -3.42812964000138
-1.55630250076729 -3.70261881259125
-1.64345267648619 -3.96240071427495
};
\addlegendentry{ScCN}
\addplot [semithick, color2, mark=+, mark size=3, mark options={solid}, only marks]
table {%
-1.47712125471966 -2.02782640400113
-1.55630250076729 -2.27692563877014
-1.64345267648619 -2.56137078912138
-1.7160033436348 -2.8071452738166
-1.79239168949825 -3.07627509506283
-1.88081359228079 -3.40542056861662
-1.96378782734556 -3.74190167158023
-2.04139268515822 -4.10325964163883
-2.12057393120585 -4.55423562275362
-2.19865708695442 -4.73580998417077
};
\addlegendentry{ScSs}
\addplot [semithick, color1, mark=asterisk, mark size=3, only marks]
table {%
-1.07918124604762 -1.75655104501726
-1.14612803567824 -2.07876798202274
-1.20411998265592 -2.35482482325706
-1.30102999566398 -2.8095153177063
-1.38021124171161 -3.17165218166161
-1.47712125471966 -3.58933299237216
};
\addlegendentry{LI GP}

\addplot [semithick, color=black, dashed]
table {%
-1.8 -4
-1.2 -1.6
};
\addlegendentry{slope 4}
\end{axis}

\end{tikzpicture} } & 
\resizebox{0.485\textwidth}{!}{
\begin{tikzpicture}

\definecolor{color0}{rgb}{0.12156862745098,0.466666666666667,0.705882352941177}
\definecolor{color1}{rgb}{1,0.498039215686275,0.0549019607843137}
\definecolor{color2}{rgb}{0.172549019607843,0.627450980392157,0.172549019607843}
\definecolor{color3}{rgb}{0.83921568627451,0.152941176470588,0.156862745098039}
\definecolor{color4}{rgb}{0.580392156862745,0.403921568627451,0.741176470588235}
\definecolor{color5}{rgb}{0.549019607843137,0.337254901960784,0.294117647058824}

\begin{axis}[
legend cell align={left},
legend columns=2,
legend style={fill opacity=0.8, draw opacity=1, text opacity=1, at={(0.03,0.03)},
  anchor=south west, draw=white!80!black},
tick align=outside,
tick pos=left,
x grid style={white!69.0196078431373!black},
xmin=-4.9139019215706, xmax=-0.99587929877438,
xlabel style={font=\color{white!15!black}},
xlabel={final error ($\ln_{10}$ scale)},
xtick style={color=black},
y grid style={white!69.0196078431373!black},
ymin=2.5, ymax=4,
ylabel style={font=\color{white!15!black}},
ylabel={CPU time ($\ln_{10}$ scale)},
ytick style={color=black}
]
\addplot [semithick, color0, mark=o, mark size=3, mark options={solid}, only marks]
table {%
-2.18741930066588 2.79423555
-2.52223857704097 2.87664791
-2.80100179070028 2.91851035
-3.2413790309208 3.04532893
-3.57008569088546 3.06581846
-3.91049376851996 3.15399199
};
\addlegendentry{LI UP}
\addplot [semithick, color3, mark=diamond, mark size=3, mark options={solid}, only marks]
table {%
-1.17397123617421 3.74774197
-1.55633961886013 3.6517426
-1.87183394397232 3.72772441
-2.13934516023454 3.70963488
-2.37088736545236 3.73920667
-2.75540685098658 3.80360092
-3.06432268093115 3.84761221
-3.42812964000138 3.89417619
-3.70261881259125 3.93436757
-3.96240071427495 3.97042711
};
\addlegendentry{ScCN}
\addplot [semithick, color2, mark=+, mark size=3, mark options={solid}, only marks]
table {%
-2.02782640400113 2.62860381
-2.27692563877014 2.71326179
-2.56137078912138 2.80678094
-2.8071452738166 2.90122686
-3.07627509506283 2.98623153
-3.40542056861662 3.0979312
-3.74190167158023 3.20013923
-4.10325964163883 3.29060311
-4.55423562275362 3.38413264
-4.73580998417077 3.46589349
};
\addlegendentry{ScSs}
\addplot [semithick, color1, mark=asterisk, mark size=3, mark options={solid}, only marks]
table {%
-1.75655104501726 2.77057035
-2.07876798202274 2.83726019
-2.35482482325706 2.89909976
-2.8095153177063 2.99024117
-3.17165218166161 3.08271639
-3.58933299237216 3.22062785
};
\addlegendentry{LI GP}
\end{axis}

\end{tikzpicture}}
\end{tabular}
\caption{Comparison of methods of order 4 applied to the NLS equation.
  On the left panel, maximal numerical error as a function of the time step (logarithmic scales); on the right panel, CPU time (in seconds) as a function of the final numerical error (logarithmic scales). }
\label{fig:ordre4_NLS1D}
\end{figure}

We compare four numerical methods of order 4 on the 1D soliton case mentioned above :
\begin{itemize}
\item the linearly implicit method of order 4 (LI UP) defined by
$(c_1,c_2,c_3,c_4)=(0,\frac13,\frac23,1)$ and
$(\lambda_1,\lambda_2,\lambda_3,\lambda_4)=(\frac{i}{2},-\frac{i}{2},\frac{i}{4},-\frac{i}{4})$.
The corresponding coefficients $(a_{i,j})_{1\leq i,j,\leq 4}$
and 
$(b_i)_{1,\leq i\leq 4}$ are computed in Example \ref{ex:4etages}. The underlying Runge--Kutta method is $\hat{A}$-stable. 

\item the linearly implicit method of order 4 (LI GP) with Gauss points defined by
$$(c_1,c_2,c_3,c_4)=\left(\frac12-\frac{\alpha_1}{70},\frac12-\frac{\alpha_2}{70}, \frac12+\frac{\alpha_2}{70},\frac12+\frac{\alpha_1}{70}\right),$$
where $\alpha_1=\sqrt{525+70\sqrt{30}}$ and $\alpha_2=\sqrt{525-70\sqrt{30}}$ and $(\lambda_1,\lambda_2,\lambda_3,\lambda_4)=(-\frac14,\frac14,-\frac12,\frac12)$. The underlying Runge--Kutta method is $\hat A$-stable (see \cite{DL2023RK}).

\item the Suzuki composition (see \eqref{Suzuki}) of the Crank-Nicolson method (ScCN) where the Crank-Nicolson method $\Phi_h: u_n \mapsto u_{n+1}$ is given by:
\begin{equation}
\label{Crank}
\dfrac{u_{n+1}-u_n}{h}=\left(\mathcal{L}+\dfrac{N(u_{n+1})+N(u_n)}{2}\right)\dfrac{u_{n+1}+u_n}{2},
\end{equation}
\item the Suzuki composition (see \eqref{Suzuki}) of the Strang splitting method (ScSs) where the Strang splitting method $\Phi_h: u_n \mapsto u_{n+1}$ is given by:
\begin{equation}
\label{Strang}
\left\{
\begin{array}{ll}
u_{n,1}=\exp\left(hN(u_n)/2\right)u_n, \\[2mm]
\left(1-\dfrac{h}{2}\mathcal{L}\right)u_{n,2}=\left(1+\dfrac{h}{2}\mathcal{L}\right)u_{n,1},\\[2mm]
u_{n+1}=\exp\left(hN(u_{n,2})/2\right)u_{n,2} .
\end{array}
\right.
\end{equation}
\end{itemize}
The Suzuki composition method that we use is given by the following numerical flow:
\begin{equation}
\label{Suzuki}
\Phi_{\alpha_3h} \circ \Phi_{\alpha_2h} \circ \Phi_{\alpha_1h},
\end{equation}
with $\alpha_1=\alpha_3=\dfrac{1}{2-\sqrt[3]{2}}$ and $\alpha_2=1-2\alpha_1$.
Since the methods $\Phi_h$ are symmetric and of order 2,
their composition \eqref{Suzuki} above is symmetric of order 4.

The results of Figure \ref{fig:ordre4_NLS1D} indicate that the four methods above have order 4 in this soliton case (left panel). This illustrates Theorem \ref{th:convergence} for the linearly implicit methods. Moreover, the two linearly implicit methods of order 4 outperform the ScCN method in terms of efficiency (computational time required for a given error) and so does the ScSs method (right panel). Indeed, the two linearly implicit methods and ScSs, as implemented above, require to solve a linear system at each time step and they display similar efficiencies within the final error range used for this simulation.

\subsubsection{A method of order 5}
\label{subsubsec:methodesordre51D}
For the illustration of the importance of the $ASI$-stability of the underlying collocation Runge--Kutta method hypothesis,
we implement the $5$-stages linearly implicit method defined in Example \ref{ex:5etagesAASpasASI}. For that linearly implicit method the underlying Runge--Kutta method is
$A$-stable, $AS$-stable, but not $ASI$-stable. We use it to solve numerically the
nonlocal nonlinear evolution PDE
\begin{equation}
  \label{eq:labelleequationdeGuillaume}
  \partial_t u = i\partial_x^2 u + u \star u \star u,
\end{equation}
with periodic boundary conditions on the torus $\R/(2\pi\Z)$ and where $\star$ denotes the convolution product.
To understand the Cauchy theory for equation \eqref{eq:labelleequationdeGuillaume},
note that the evolution of the Fourier coefficient $c_k(u)$ of order $k\in\Z$ is such that
\begin{equation}
  \label{eq:uk2}
  \left(c_k(u(t,\cdot)\right)^2 = \frac{(c_k(u_0))^2}{\left(1-i\frac{(c_k(u_0))^2}{k^2}\right)
  {\rm e}^{2ik^2 t}+i\frac{(c_k(u_0))^2}{k^2}},
\end{equation}
where $u_0$ is the initial datum associated to \eqref{eq:labelleequationdeGuillaume} at $t=0$.
Assuming that $u_0\in L^2(\R/(2\pi\Z))$ is real-valued and even,
we have that for all $k\in\Z$, $c_k(u_0)\in\R$
and taking appropriate square roots in \eqref{eq:uk2}, and forming the corresponding Fourier
series, we obtain a solution to \eqref{eq:labelleequationdeGuillaume} for all $t\in\R$.
We plot in Figure \ref{fig:ordre5_NLSbizarre_pasASIstable}
the discrete $L^2$-norm of the difference at times close to $T=1.0$ between
the exact solution of \eqref{eq:labelleequationdeGuillaume} defined using \eqref{eq:uk2}
and the numerical solution obtained by the non-ASI method described in Example \ref{ex:5etagesAASpasASI}.
We take $N=30$ Fourier modes (from $k=0$ to $k=29$) and $c_k(u_0)={\rm e}^{-|k|}$.
We chose time steps between $1.0/2^3$ and $1.0/2^8$
and we also consider time steps satisfying $1-hk^2\alpha=0$,
where $\alpha=3\sqrt{7}/56$ (see the analysis of this method carried out in \cite{DL2023RK}). For these time steps, the matrix
$(I-hA\otimes L)$ (for diagonal $L$ with entries $-ik^2$) is not invertible.
We observe in Figure \ref{fig:ordre5_NLSbizarre_pasASIstable} that the method displays a convergence of order 5 for regular time steps,
which is destroyed for resonant time steps of the form $1.0/(\alpha k^2)$.
This illustrates the importance of the hypothesis of $ASI$-stability (included in $\hat{A}$-stability)
in Theorem \ref{th:convergence} for the convergence of linearly implicit methods
applied to a nonlinear Schr\"odinger evolution PDE.

\begin{figure}[!h]
\centering
\resizebox{0.48\textwidth}{!}{
\begin{tikzpicture}

\definecolor{color0}{rgb}{0.172549019607843,0.627450980392157,0.172549019607843}
\definecolor{color1}{rgb}{1,0.498039215686275,0.0549019607843137}

\begin{axis}[
legend cell align={left},
legend style={fill opacity=0.8, draw opacity=1, text opacity=1, draw=white!80!black},
tick align=outside,
tick pos=left,
x grid style={white!69.0196078431373!black},
xmin=-2.16741596878066, xmax=-0.842883987859147,
xtick style={color=black},
xtick={-2.2,-2,-1.8,-1.6,-1.4,-1.2,-1,-0.8},
xticklabels={--2.2,--2.0,--1.8,--1.6,--1.4,--1.2,--1.0,-0.8},
xlabel={log10 of the time step},
y grid style={white!69.0196078431373!black},
ymin=-7.49494046646252, ymax=18.0661236476368,
ytick style={color=black},
ylabel={log10 of the error},
]
\addplot [semithick, color0, mark=*, mark size=3, mark options={solid}, only marks]
table {%
-0.903089986991944 -0.335164541895539
-1.20411998265592 -1.82975212802136
-1.50514997831991 -3.34186964805822
-1.80617997398389 -4.84288251999412
-2.10720996964787 -6.33307391582164
};
\addlegendentry{regular steps}
\addplot [semithick, color0, dashed]
table {%
-0.903089986991944 0.484550065040282
-1.20411998265592 -1.02059991327962
-1.50514997831991 -2.52574989159953
-1.80617997398389 -4.03089986991944
-2.10720996964787 -5.53604984823934
};
\addlegendentry{slope 5}
\addplot [semithick, color1, mark=asterisk, mark size=3, mark options={solid}, only marks]
table {%
-0.957662221704478 3.31345883955977
-1.05996726659924 3.56305618175698
-1.15148224772059 4.11438775152909
-1.23426761803704 4.91811112192611
-1.30984473981584 5.59971861801783
-1.37936895233426 6.00494398799987
-1.44373831907707 8.02967651778697
-1.50366476583195 6.86244750880081
-1.55972221303244 7.37927394069075
-1.61238009047714 10.854584198146
-1.6620272579272 9.21190007457814
-1.70898944962625 11.187381059741
-1.75354223904855 12.1791730413851
-1.79592083718843 12.9568380430226
-1.836327609365 13.357188568854
-1.87493791975578 14.2451986333748
-1.9119047311438 14.6999564301633
-1.94736226506467 15.389594933047
-1.98142894366223 15.9051330424637
-2.01420977603857 16.3513696710061
-2.04579831040503 16.9042570969959
};
\addlegendentry{resonant steps}
\end{axis}

\end{tikzpicture}}
\caption{Numerical $L^2$-error at time $T=1.0$ as a function of the time step for
  the solution to Equation \eqref{eq:labelleequationdeGuillaume} with initial datum given in the
  text through its Fourier coefficients,
  using the non-$ASI$ linearly implicit method of order 5 described in Example \ref{ex:5etagesAASpasASI}.
  We plot regular time steps of the form $1.0/2^p$ for $p\in\{3,\cdots,8\}$ (green dots),
  a regular line of slope 5 (dashed green line),
  and resonant time steps of size $1.0/(\alpha k^2)$ for $1\leq k\leq 30$.}
\label{fig:ordre5_NLSbizarre_pasASIstable}
\end{figure}
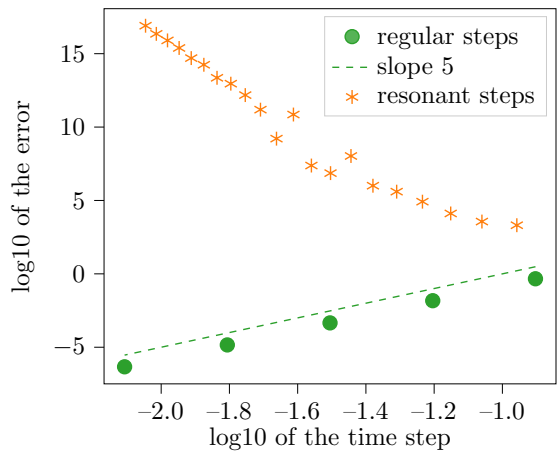

\begin{remark}
The numerical examples above (in Figure \ref{fig:ordre2_NLS1D_pasAstable} of a method which is not $A$-stable
(see Section \ref{subsubsec:methodesordre21D}) and in Figure \ref{fig:ordre5_NLSbizarre_pasASIstable} of a method which is $A$, $AS$ and not $ASI$ stable (see Section \ref{subsubsec:methodesordre51D})), show that linearly implicit methods may fail to converge for evolution PDE problems, even if they would converge for evolution ODE problems. Moreover, these two numerical examples demonstrate the relevance of the concepts of $A$-stability and $ASI$-stability in the hypotheses of Theorem \ref{th:convergence}. 
\end{remark}

\subsection{Numerical experiments in 2 dimensions for the NLS equation}
\label{subsec:numNLS2D}

\subsubsection{Space discretization}
We consider the focusing NLS equation
\begin{equation}
  \label{eq:NLSstarshaped}
  i \partial_t u = -\Delta u - q|u|^2 u,
\end{equation}
on a star-shaped domain $\mathcal S\subset \R^2$ with homogeneous Dirichlet boundary conditions,
with $q\in\R$. The energy associated to \eqref{eq:NLSstarshaped} reads
\begin{equation*}
  E(u) = \frac12 \int_{\mathcal S} |\nabla u|^2{\rm d}x {\rm d}y
  - \frac{q}{4} \int_{\mathcal S} |u|^4 {\rm d}x {\rm d}y.
\end{equation*}
The domain $\mathcal S$ consists in the union of the interior of the triangle with vertices
\begin{equation*}
  \left(R\sin\left(2k\frac{\pi}{3}\right),R\cos\left(2k\frac{\pi}{3}\right)\right),
  \qquad
  0\leq k \leq 2,
\end{equation*}
and that of the triangle with vertices
\begin{equation*}
  \left(R\sin\left((2k+1)\frac{\pi}{3}\right),R\cos\left((2k+1)\frac{\pi}{3}\right)\right),
  \qquad
  0\leq k \leq 2,
\end{equation*}
for some $R>0$.

The discretization of $\mathcal S$ is carried out using {\tt GMSH} by generating an admissible
(in the sense of finite volumes) triangular mesh of $\mathcal S$.
The triangles are denoted $({\mathcal T}_j)_{1\leq j\leq J}$ for some $J\geq 1$.
The set of the corresponding edges is denoted by $\mathcal E$, and is partitioned in the
set $\mathcal E^{\rm int}$ of interior edges (belonging to 2 triangles)
and the set $\mathcal E^{\rm ext}$ of exterior edges (belonging to 1 triangle).
For all triangle $\mathcal T_j$, we denote by $\mathcal E_j\subset\mathcal E$ the set
of its 3 edges.
Denoting by $U\in\C^J$ a finite volume approximation of a complex-valued function $u$
over $\mathcal S$, we define a $J\times J$ matrix $L$ by setting for all $k\in\{1,\cdots,J\}$
\begin{equation*}
  (LU)_k = \frac{i}{m\left(\mathcal T_k\right)}\left(\sum_{\sigma\in \mathcal E_k\cap\mathcal E^{\rm int}} \frac{U_j-U_k}{d_{\sigma}} \ell_{\sigma}
  + \sum_{\sigma \in\mathcal E_k\cap\mathcal E^{\rm ext}} \frac{0-U_k}{d_{\sigma}} \ell_{\sigma}\right),
\end{equation*}
where
\begin{itemize}
\item in the first sum, for $\sigma\in \mathcal E_k\cap\mathcal E^{\rm int}$, the letter
  $j$ denotes {\it the} integer in $\{1,\cdots,J\}\setminus\{k\}$
  such that $\sigma$ is both an edge of
  $\mathcal T_j$ and $\mathcal T_k$, $d_\sigma$ is the distance between the centers of mass
  of $\mathcal T_j$ and $\mathcal T_k$ and $\ell_\sigma$ is the length of the edge $\sigma$;
\item in the second sum, for $\sigma\in \mathcal E_k\cap\mathcal E^{\rm ext}$,
      $d_\sigma$ denotes the distance between the center of mass of $\mathcal T_k$ and
      the edge $\sigma$, and $\ell_\sigma$ still denotes the length of the edge $\sigma$;
\item the area of the triangle $\mathcal T_k$ is denoted by $m\left(\mathcal T_k\right)$. In our case all the triangles share the same area $m(\mathcal{T})$. 
\end{itemize}
This way, $LU$ is a finite volume approximation of $i\Delta u$. It is easy to check
that the matrix $L$ is skew-symmetric, with spectrum included in $i\R^-$.
From now on, we consider the following semidiscretization of \eqref{eq:NLSstarshaped}:
\begin{equation}
  \label{eq:NLSstarshapeddiscresp}
  U'(t)=LU(t)+N(U)\ast U,
\end{equation}
where $N(U)\in\R^J$ has component $j$ equal to $iq|U_j|^2$ and $\ast$ stands for the componentwise
multiplication in $\C^J$.

\subsubsection{Comparison of methods of order 2 on the star shaped domain}
We consider the case $q=1$, $R=1$ with the initial datum $U_0$ as the interpolation of 
\begin{equation*}
u_0(x,y)=\sqrt{\dfrac{98}{\pi}}\exp\left(-49(x^2+y^2)\right)\exp\left(-20i x\right),
\end{equation*}
at the centers of the triangles. We set $T=0.1$ as the final time. and we consider.
An example of the dynamics is presented in Figure \ref{fig:ordre2_NLS2D_dyna}
with $J=161472$ triangles.

Note that, since the spectrum of $L$ is not explicitly known, the computation of
exponentials involving $L$ requires approximations. Depending on where these approximations
are used 
(in splitting methods, in {\it e.g.} Lawson methods, or in exponential Runge--Kutta methods),
the order of approximation has to  be chosen carefully.
Moreover, when $J$ is big, computing these approximations can be costly.

\begin{figure}[!h]
\centering
\begin{tabular}{ccc}
	\resizebox{0.30\textwidth}{!}{\includegraphics{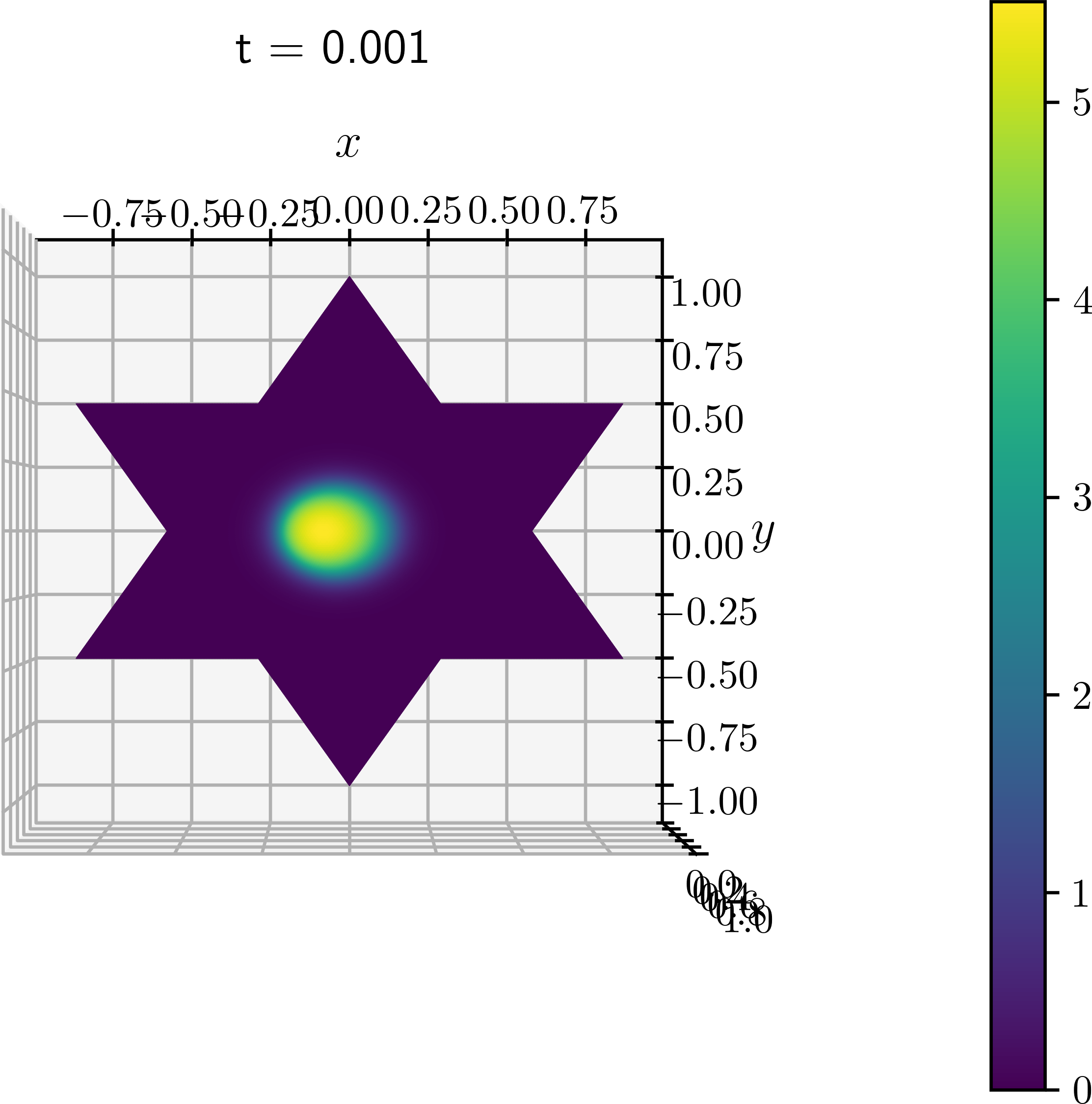}} & 
    \resizebox{0.30\textwidth}{!}{\includegraphics{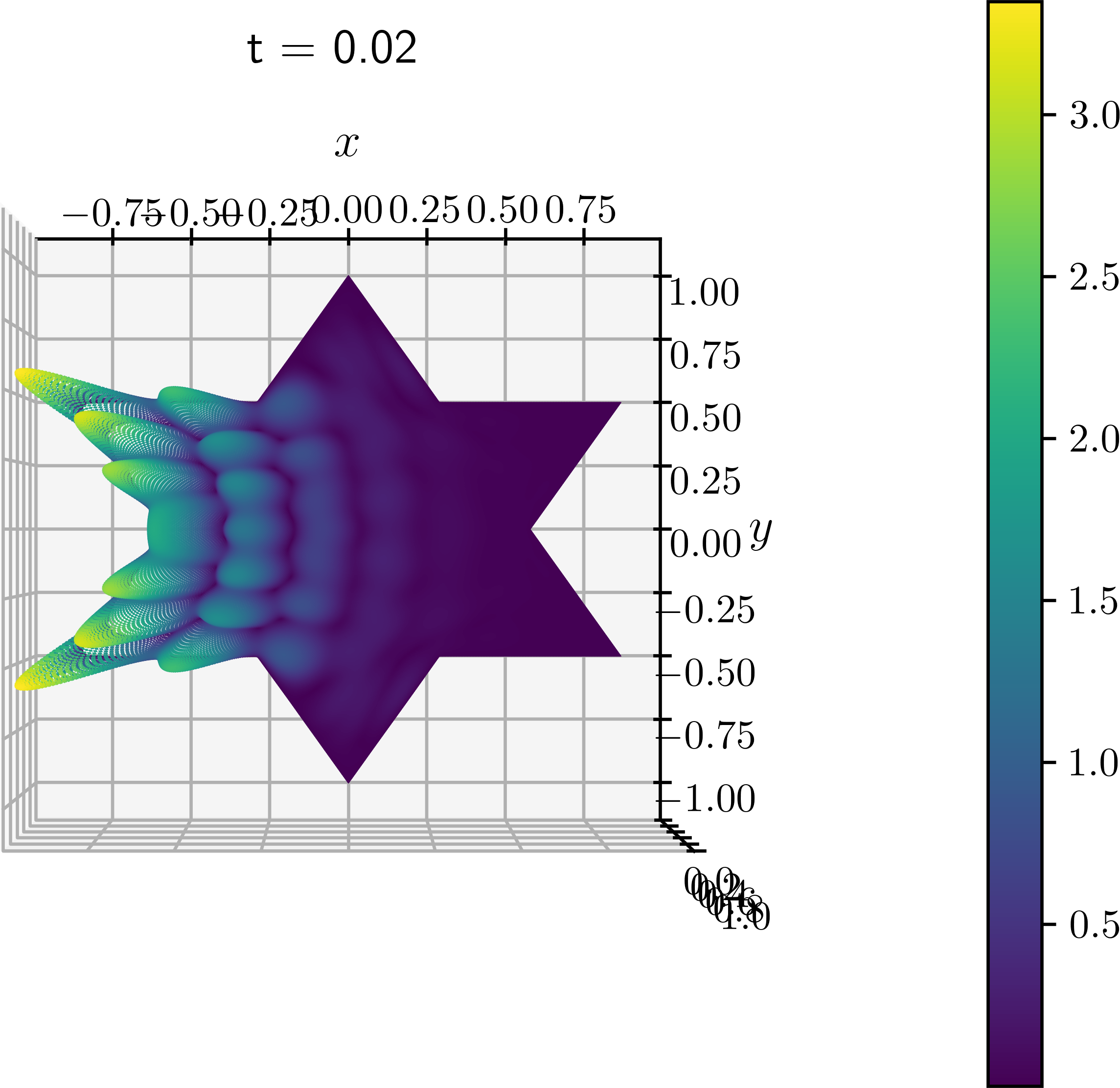}}& 
    \resizebox{0.30\textwidth}{!}{\includegraphics{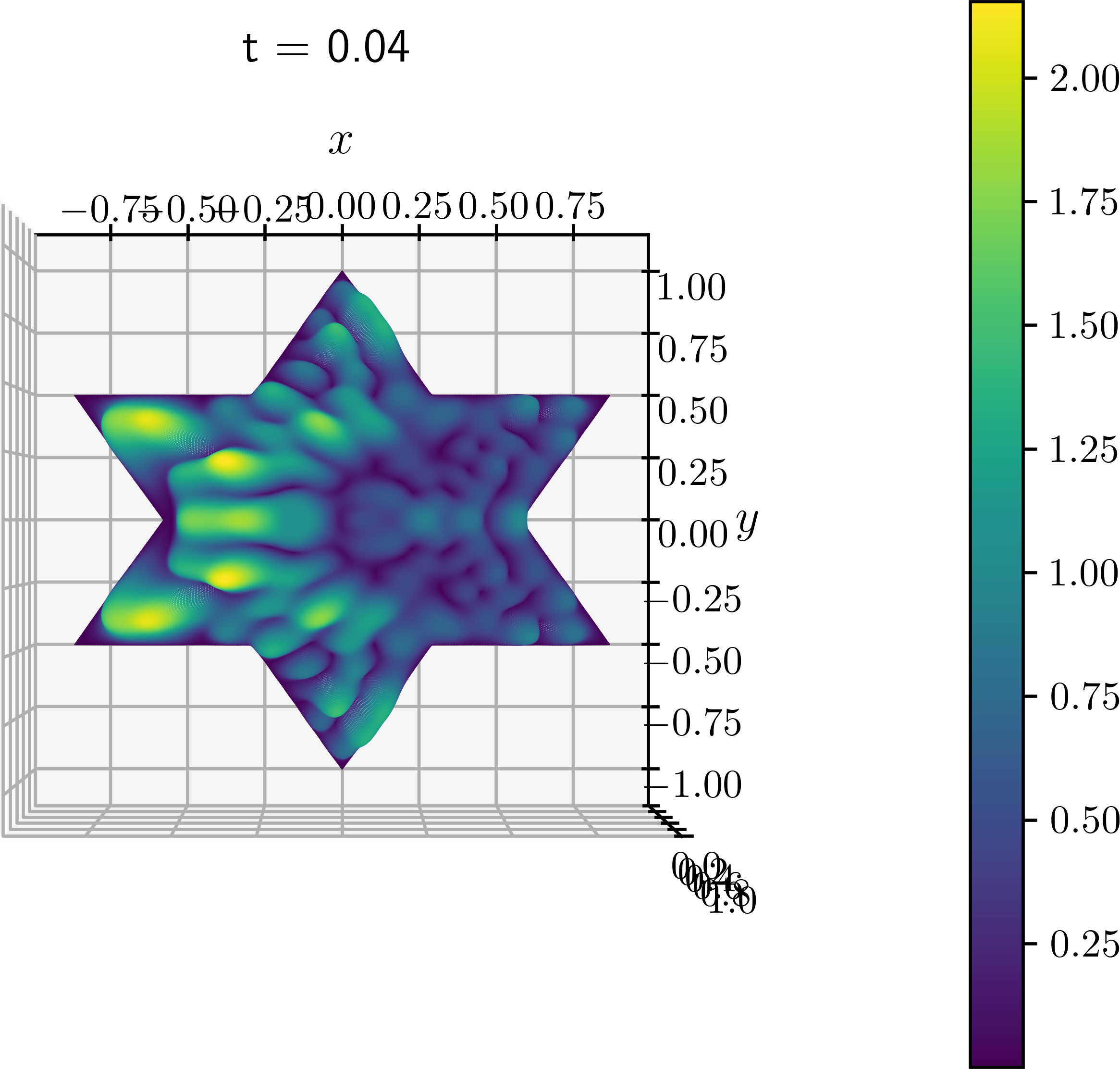}}
    \\
    \resizebox{0.30\textwidth}{!}{\includegraphics{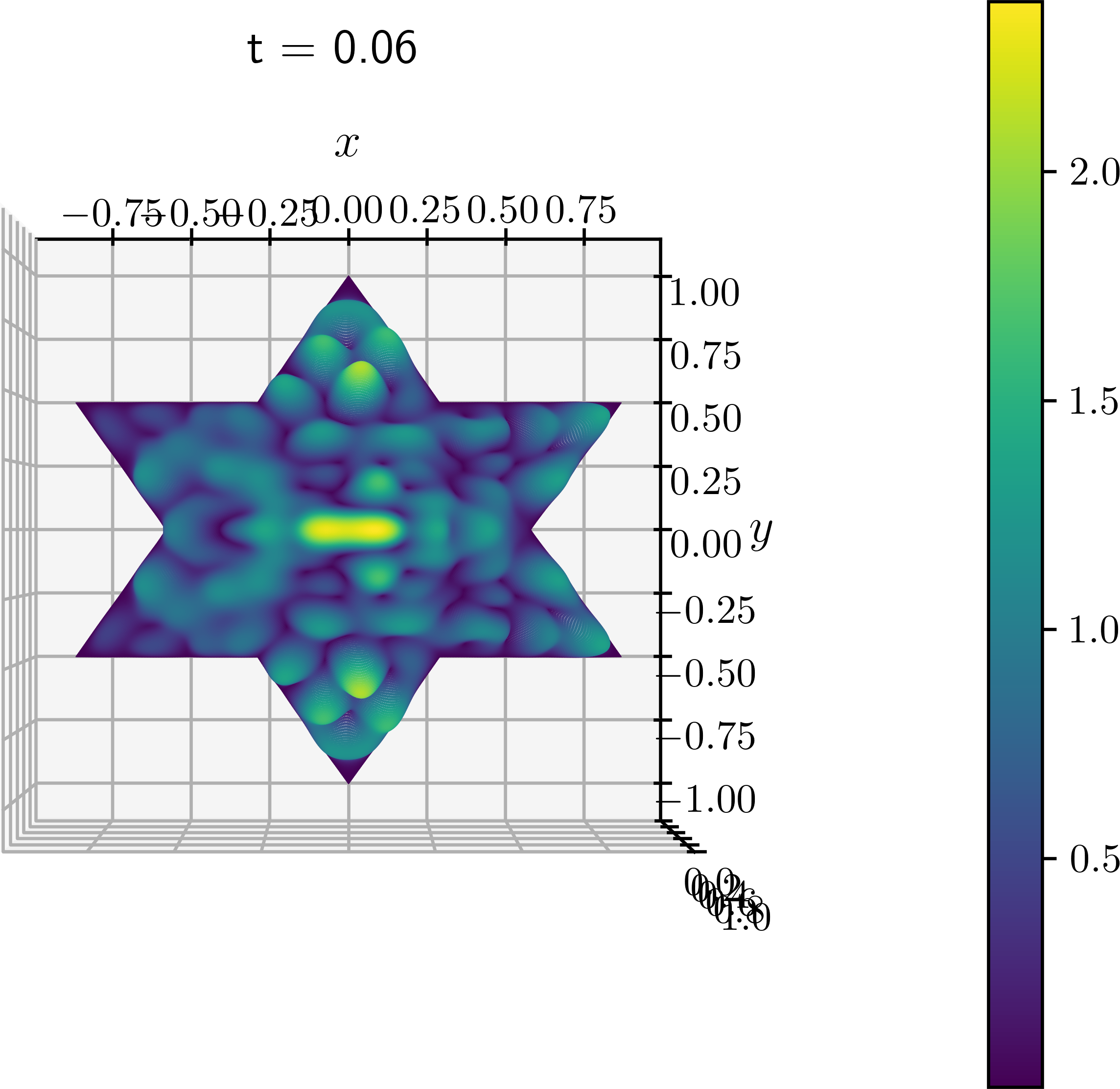}}&
    \resizebox{0.30\textwidth}{!}{\includegraphics{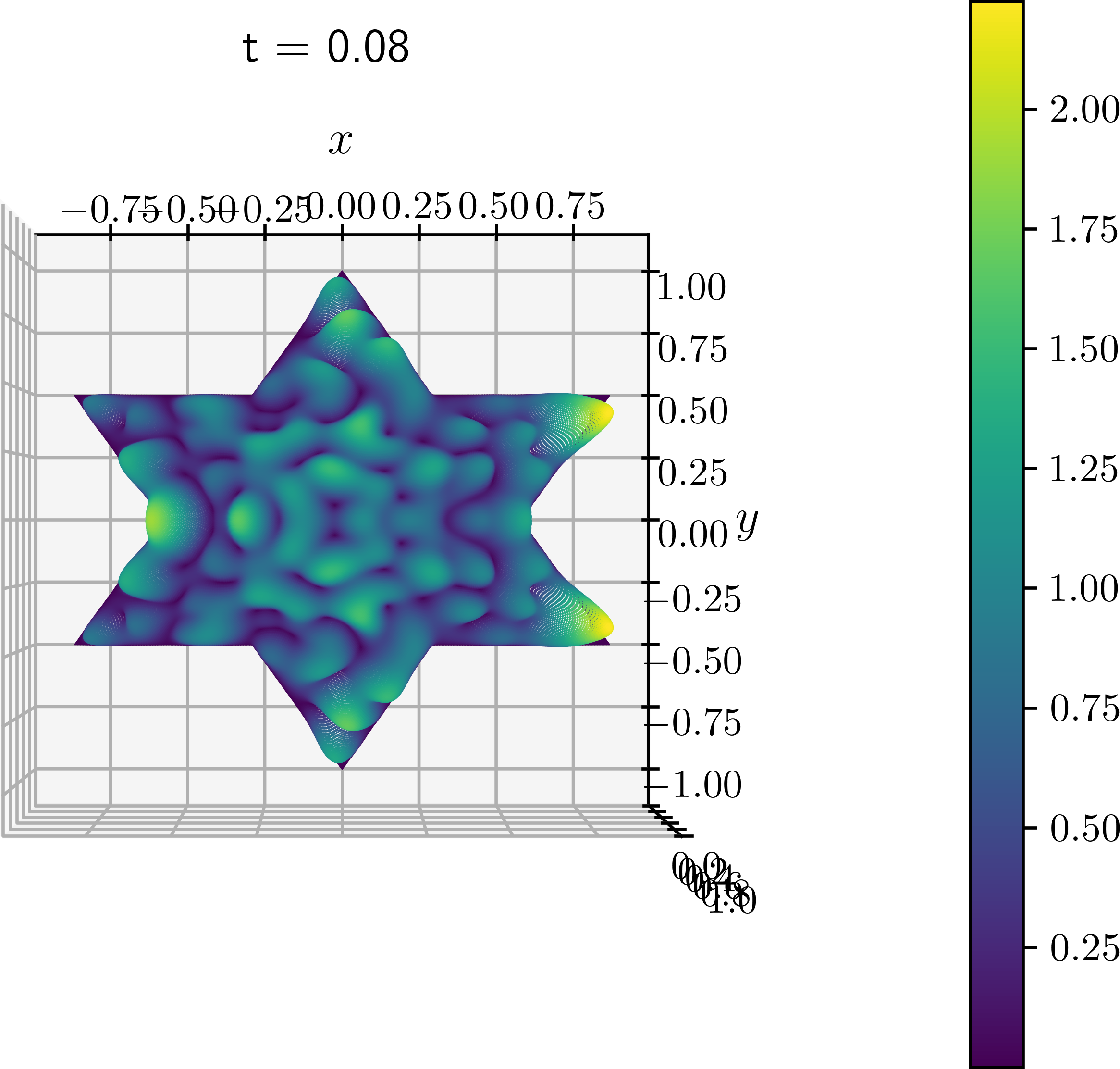}}&
    \resizebox{0.30\textwidth}{!}{\includegraphics{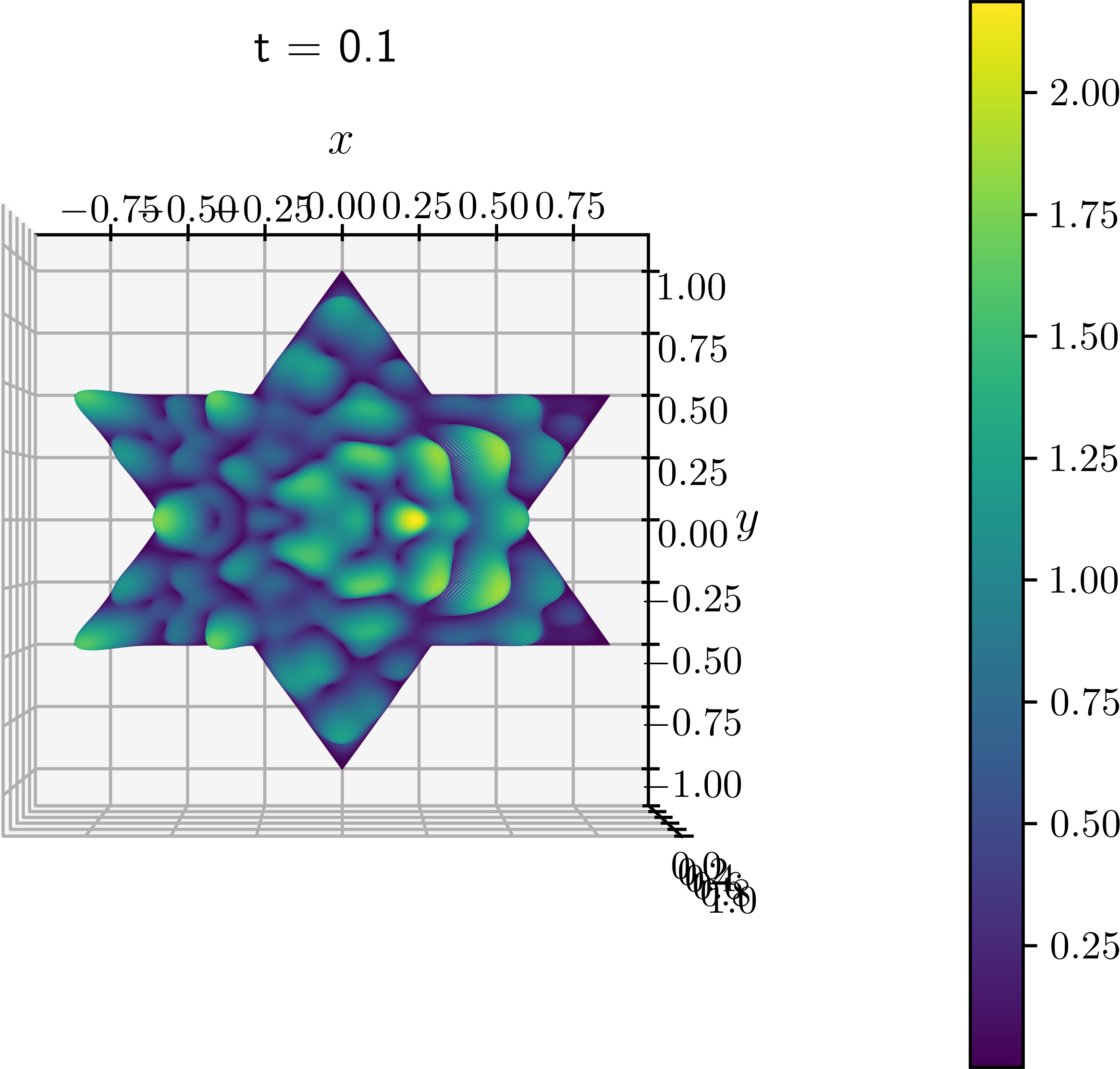}}
\end{tabular}
\caption{The dynamics of the test case of the NLS equation \eqref{eq:NLSstarshaped} on a star-shaped domain.}
\label{fig:ordre2_NLS2D_dyna}
\end{figure}

We present in Figure \ref{fig:ordre2_NLS2D} a numerical comparison with $J=10\ 092$ of methods expected to be of order 2 for this problem. We show at the same time the final error as a function of the time step (left panel) and the CPU time as a function of the final error (right panel), both in log scales. The four methods we consider are
\begin{itemize}
	\item the Strang splitting method with $(Id+hL/2)(Id-hL/2)^{-1}$ as an approximation of $e^{hL}$ in a classical and efficient implementation involving only the Lie splitting method,
	\item the Crank--Nicolson method (C-N),
	\item the linearly implicit method with uniform points (LI UP) based on the Runge--Kutta method of example \ref{ex:meth2etagesLobatto} with $\lambda_1=-\lambda_2=1/2$,
	\item the linearly implicit method with Gauss points (LI GP) based on the Runge--Kutta method of example \ref{ex:meth2etagesGauss} with $\lambda_1=-\lambda_2=1/2$.
\end{itemize}
We implement the four methods in {\tt Python} and use the {\tt spsolve} function of {\tt scipy.sparse.linalg} to compute the solutions to all the linear systems. 
We use a reference solution computed with a method of order 5 with a very small time step to compute the final errors. The three first methods have numerical order 2 as predicted by Theorem \ref{th:convergence}. The fourth one seems to have order 3 and hence to behave even better than predicted by Theorem \ref{th:convergence}. In terms of efficiency, the linearly implicit method with Gauss points outperforms even the explicit Strang splitting method describe above. However, the Strang splitting method is faster than the linearly implicit method with uniform points.

\begin{figure}[!h]
\centering
\begin{tabular}{cc}
\resizebox{0.48\textwidth}{!}{
\begin{tikzpicture}

\definecolor{color0}{rgb}{0.12156862745098,0.466666666666667,0.705882352941177}
\definecolor{color1}{rgb}{1,0.498039215686275,0.0549019607843137}
\definecolor{color2}{rgb}{0.172549019607843,0.627450980392157,0.172549019607843}
\definecolor{color3}{rgb}{0.83921568627451,0.152941176470588,0.156862745098039}
\definecolor{color4}{rgb}{0.580392156862745,0.403921568627451,0.741176470588235}
\definecolor{color5}{rgb}{0.549019607843137,0.337254901960784,0.294117647058824}

\begin{axis}[
legend cell align={left},
legend columns=3,
legend style={
  fill opacity=0.8,
  draw opacity=1,
  text opacity=1,
  at={(0.03,0.03)},
  anchor=south west,
  draw=white!80!black
},
tick align=outside,
tick pos=left,
x grid style={white!69.0196078431373!black},
xlabel={time step \(\displaystyle h\) (\(\displaystyle \ln_{10}\) scale)},
xmin=-4.2, xmax=-2.6,
xtick style={color=black},
y grid style={white!69.0196078431373!black},
ylabel={final error (\(\displaystyle \ln_{10}\) scale)},
ymin=-2.2, ymax=-0.4,
ytick style={color=black}
]
\addplot [semithick, color0, mark=o, mark size=3, mark options={solid,fill opacity=0}, only marks]
table {%
-3.60205999132796 -0.604258919855249
-3.69897000433602 -0.773342076438952
-3.77815125038364 -0.920163990532986
-3.84509804001426 -1.04826687394716
-3.90308998699194 -1.16111447046157
-3.95424250943932 -1.26158227378437
-4 -1.35194696879474
};
\addlegendentry{LI UP}
\addplot [semithick, color1, mark=asterisk, mark size=3, mark options={solid,fill opacity=0}, only marks]
table {%
-2.77815125038364 -0.432374825439581
-2.90308998699194 -0.697469675660379
-3 -0.957834750757471
-3.07918124604762 -1.20471388710688
-3.14612803567824 -1.43169941693913
-3.20411998265592 -1.63589718503871
};
\addlegendentry{LI GP}
\addplot [semithick, color2, mark=+, mark size=3, mark options={solid,fill opacity=0}, only marks]
table {%
-3.60205999132796 -0.602682478363296
-3.69897000433602 -0.771714954274549
-3.77815125038364 -0.918494406421278
-3.84509804001426 -1.04656435623373
-3.90308998699194 -1.1593869910545
-3.95424250943932 -1.25983583537982
-4 -1.35018662112496
};
\addlegendentry{Strang}
\addplot [semithick, color3, mark=diamond, mark size=3, mark options={solid,fill opacity=0}, only marks]
table {%
-3.60205999132796 -0.604927112182141
-3.69897000433602 -0.77407514460598
-3.77815125038364 -0.920949197918535
-3.84509804001426 -1.04909108991825
-3.90308998699194 -1.16196753265367
-3.95424250943932 -1.26245649841949
-4 -1.35283682032462
};
\addlegendentry{C-N}
\addplot [semithick, color4, dashed]
table {%
-3.60205999132796 -0.804119982655925
-3.69897000433602 -0.997940008672037
-3.77815125038364 -1.15630250076729
-3.84509804001426 -1.29019608002851
-3.90308998699194 -1.40617997398389
-3.95424250943932 -1.50848501887865
-4 -1.6
};
\addlegendentry{slope 2}
\addplot [semithick, color5, dashed]
table {%
-2.87815125038364 -0.434453751150931
-3.00308998699194 -0.809269960975833
-3.1 -1.1
-3.17918124604762 -1.33754373814287
-3.24612803567824 -1.53838410703472
-3.30411998265592 -1.71235994796777
};
\addlegendentry{slope 3}
\end{axis}

\end{tikzpicture}
  } & 
      \resizebox{0.47\textwidth}{!}{
\begin{tikzpicture}

\definecolor{color0}{rgb}{0.12156862745098,0.466666666666667,0.705882352941177}
\definecolor{color1}{rgb}{1,0.498039215686275,0.0549019607843137}
\definecolor{color2}{rgb}{0.172549019607843,0.627450980392157,0.172549019607843}
\definecolor{color3}{rgb}{0.83921568627451,0.152941176470588,0.156862745098039}

\begin{axis}[
legend cell align={left},
legend columns=2,
legend style={
  fill opacity=0.8,
  draw opacity=1,
  text opacity=1,
  at={(0.03,0.03)},
  anchor=south west,
  draw=white!80!black
},
tick align=outside,
tick pos=left,
x grid style={white!69.0196078431373!black},
xlabel={final error (\(\displaystyle \ln_{10}\) scale)},
xmin=-1.69607330301867, xmax=-0.372198707459625,
xtick style={color=black},
y grid style={white!69.0196078431373!black},
ylabel={CPU time (\(\displaystyle \ln_{10}\) scale)},
ymin=0.5, ymax=2.7,
ytick style={color=black}
]
\addplot [semithick, color0, mark=o, mark size=4, mark options={solid,fill opacity=0}, only marks]
table {%
-0.604258919855249 1.59579515940592
-0.773342076438952 1.70192053499757
-0.920163990532986 1.79180014314963
-1.04826687394716 1.84929494296289
-1.16111447046157 1.91364569488819
-1.26158227378437 1.96659727148672
-1.35194696879474 2.03656406216087
};
\addlegendentry{LI UP}
\addplot [semithick, color1, mark=asterisk, mark size=4, mark options={solid,fill opacity=0}, only marks]
table {%
-0.432374825439581 1.2138944400939
-0.697469675660379 1.34116884834393
-0.957834750757471 1.43776976865051
-1.20471388710688 1.52407671871099
-1.43169941693913 1.57282452428519
-1.63589718503871 1.63070942840983
};
\addlegendentry{LI GP}
\addplot [semithick, color2, mark=+, mark size=4, mark options={solid,fill opacity=0}, only marks]
table {%
-0.602682478363296 1.40159168179196
-0.771714954274549 1.4856378849348
-0.918494406421278 1.56554522673983
-1.04656435623373 1.64684545106518
-1.1593869910545 1.70945509152791
-1.25983583537982 1.75666669994016
-1.35018662112496 1.80269927657342
};
\addlegendentry{Strang}
\addplot [semithick, color3, mark=diamond, mark size=4, mark options={solid,fill opacity=0}, only marks]
table {%
-0.604927112182141 2.11189950484408
-0.77407514460598 2.20100482654455
-0.920949197918535 2.27311787883894
-1.04909108991825 2.34077588260497
-1.16196753265367 2.41420133242413
-1.26245649841949 2.45802974700726
-1.35283682032462 2.56599461204042
};
\addlegendentry{C-N}
\end{axis}

\end{tikzpicture}
      }
\end{tabular}
\caption{Comparison for $J=10\ 092$ of methods of order 2 applied to the NLS equation on a star-shaped domain.
  On the left panel, final numerical error as a function of the time step (logarithmic scales);
  on the right panel, CPU time (in seconds) as a function of the final numerical error
  (logarithmic scales). }
\label{fig:ordre2_NLS2D}
\end{figure}
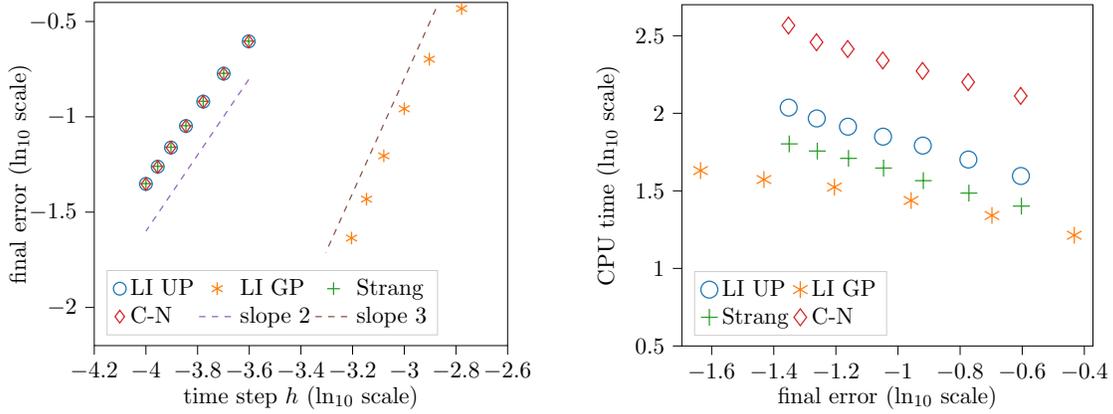

We compare in Figure \ref{fig:mass_energy_2_NLS2D} the preservation property of the mass (squared $L^2$-norm) and energy 
\begin{equation}
\label{eq:energyNLS}
E(U)=\dfrac{m(\mathcal{T})}{2}U^\star LU-q\dfrac{m(\mathcal{T})}{4} \sum_{k=1}^J|U_k|^4,
\end{equation} 
where $U^\star$ is the conjugate transpose of $U$. As expected, the Crank-Nicolson method and Strang splitting method preserve the mass. So does the linearly implicit method with Gauss points since it satisfies the Cooper condition \eqref{eq:coopercondition}. This is coherent with Proposition \ref{prop:Coopercondition}. The linearly implicit method with uniform points does not preserve mass as it does not satisfy the Cooper condition. Among the four methods of order 2, the Crank--Nicolson method is the only one preserving the energy.

\begin{figure}[!h]
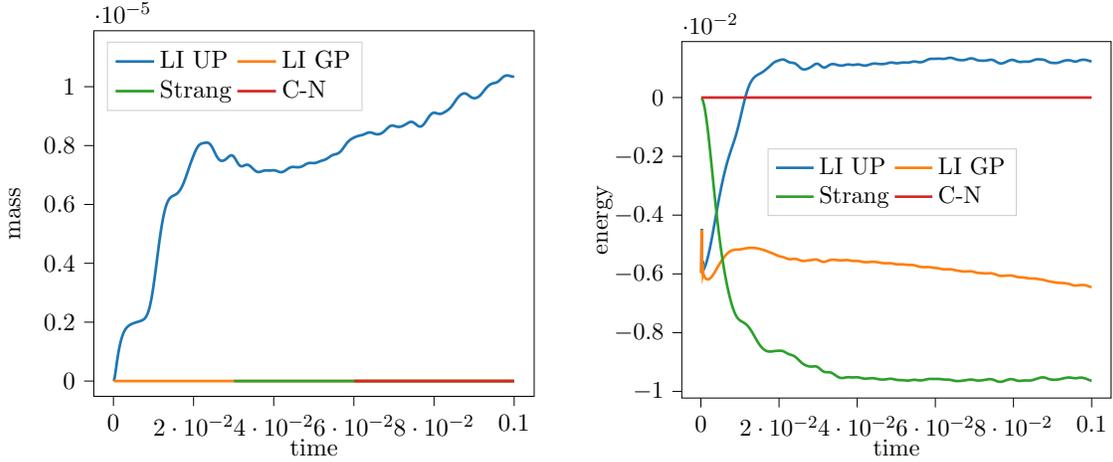

\centering
\begin{tabular}{cc}
\resizebox{0.48\textwidth}{!}{
  \input{L2carre2etages-NT=1000}
  } & 
      \resizebox{0.47\textwidth}{!}{
      \input{NRJ2etages-NT=1000}
      }
\end{tabular}
\caption{Deviation of the mass (left panel) and energy (right panel) from their initial values by the four methods of order 2 for $J=10\ 092$ and $h=10^{-4}$}
\label{fig:mass_energy_2_NLS2D}
\end{figure}

\subsection{Numerical experiments in 1 dimension for the NLH equation}
\label{subsec:numNLH1D}

The comparaison of methods of order 1 has been carried out in \cite{DL2020}. We focus here on methods of order 2. 

In this section, we perform numerical experiments in dimension 1 on $\Omega=(-50,50)$ with homogeneous Dirichlet boundary conditions, on the equation
\begin{equation}
\label{eq:chaleur}
\partial_t u(t,x)=\partial_{xx}^2 u(t,x)+u(t,x)^3,
\end{equation}
starting from
\begin{equation*}
u_0(x)=\frac12 \sin\left(\frac{\pi}{100}x+\frac{\pi}{2}\right),
\end{equation*}
for $t \in (0,T)$ with $T=1.0$. We use classical finite differences in space with $2^{14}$ points. We compare two linearly implicit methods (with uniform (Example \ref{ex:meth2etagesLobatto}) or Gauss points (Example \ref{ex:meth2etagesGauss}), the eigenvalues of the matrix $D$ being $\pm \frac12$ in both cases) with the classical methods of Crank-Nicolson and Strang splitting. A reference solution is computed with a classical Runge--Kutta method of order 10 with a small time step. For the initialization of the linearly implicit methods, we use one step of Strang splitting starting from $u_0$ over $[0,c_1h]$, $[0,c_2h]$ and $[0,h]$ and then run the method (see Section \ref{subsec:initialization}). Numerical results are displayed in Figure \ref{fig:chaleurordre2}.

\begin{figure}[!h]
\centering
\begin{tabular}{cc}
\resizebox{0.48\textwidth}{!}{
\begin{tikzpicture}

\definecolor{color0}{rgb}{0.12156862745098,0.466666666666667,0.705882352941177}
\definecolor{color1}{rgb}{1,0.498039215686275,0.0549019607843137}
\definecolor{color2}{rgb}{0.172549019607843,0.627450980392157,0.172549019607843}
\definecolor{color3}{rgb}{0.83921568627451,0.152941176470588,0.156862745098039}
\definecolor{color4}{rgb}{0.580392156862745,0.403921568627451,0.741176470588235}
\definecolor{color5}{rgb}{0.549019607843137,0.337254901960784,0.294117647058824}

\begin{axis}[
legend cell align={left},
legend columns=2,
legend style={
  fill opacity=0.8,
  draw opacity=1,
  text opacity=1,
  at={(0.57,0.03)},
  anchor=south west,
  draw=white!80!black
},
tick align=outside,
tick pos=left,
x grid style={darkgray176},
xlabel={time step \(\displaystyle h\) (\(\displaystyle \ln_{10}\) scale)},
xmin=-2.67131197061447, xmax=-1.00336549725873,
xtick style={color=black},
y grid style={darkgray176},
ylabel={final error (\(\displaystyle \ln_{10}\) scale)},
ymin=-9.46457308108112, ymax=-2.12731531619731,
ytick style={color=black}
]
\addplot [semithick, color0, mark=o, mark size=3, mark options={solid,fill opacity=0}, only marks]
table {%
-1.07918124604762 -2.5692592917292
-1.14612803567824 -2.68979636339957
-1.20411998265592 -2.79567769849672
-1.30102999566398 -2.9752088503594
-1.38021124171161 -3.12393483816584
-1.47712125471966 -3.30801442940528
-1.55630250076729 -3.45981787629456
-1.64345267648619 -3.62810486702793
-1.7160033436348 -3.76901247565403
-1.79239168949825 -3.91804646498948
-1.88081359228079 -4.0912841878078
-1.96378782734556 -4.25444058970553
-2.04139268515822 -4.407470874872
-2.12057393120585 -4.56397022437071
-2.19865708695442 -4.71859889178236
-2.27875360095283 -4.87747370708313
-2.35793484700045 -5.03474899970472
-2.43775056282039 -5.19346610184896
-2.51587384371168 -5.34896537993161
-2.59549622182557 -5.50757411991718
};
\addlegendentry{LI UP}

\addplot [semithick, color3, mark=diamond, mark size=3, mark options={solid,fill opacity=0}, only marks]
table {%
-1.07918124604762 -3.03751436401933
-1.14612803567824 -3.171648743364
-1.20411998265592 -3.287788810815
-1.30102999566398 -3.48179240811714
-1.38021124171161 -3.64025461910744
-1.47712125471966 -3.83415624593029
-1.55630250076729 -3.99256305416578
-1.64345267648619 -4.16689672713239
-1.7160033436348 -4.31201710975371
-1.79239168949825 -4.46480801642277
-1.88081359228079 -4.64166396249949
-1.96378782734556 -4.80762055183278
-2.04139268515822 -4.96283568512967
-2.12057393120585 -5.12120210497736
-2.19865708695442 -5.27737037551255
-2.27875360095283 -5.43756408517562
-2.35793484700045 -5.59592704979606
-2.43775056282039 -5.75555880716131
-2.51587384371168 -5.91180559563598
-2.59549622182557 -6.07105048095644
};
\addlegendentry{C-N}
\addplot [semithick, color2, mark=+, mark size=3, mark options={solid,fill opacity=0}, only marks]
table {%
-1.07918124604762 -6.10182585943264
-1.14612803567824 -6.23564898767978
-1.20411998265592 -6.35158368802422
-1.30102999566398 -6.54534931277796
-1.38021124171161 -6.70366419062781
-1.47712125471966 -6.89745287699438
-1.55630250076729 -7.05578841370453
-1.64345267648619 -7.23008793608839
-1.7160033436348 -7.37510157454337
-1.79239168949825 -7.52785839158715
-1.88081359228079 -7.70477601459009
-1.96378782734556 -7.87056834573293
-2.04139268515822 -8.02538563415332
-2.12057393120585 -8.18389079535676
-2.19865708695442 -8.34029072684988
-2.27875360095283 -8.49972311250433
-2.35793484700045 -8.65787514923195
-2.43775056282039 -8.81718315992945
-2.51587384371168 -8.97179362255891
-2.59549622182557 -9.12401225258933
};
\addlegendentry{Strang}

\addplot [semithick, color1, mark=asterisk, mark size=3, mark options={solid,fill opacity=0}, only marks]
table {%
-1.07918124604762 -2.48715659578116
-1.14612803567824 -2.60744639819277
-1.20411998265592 -2.71316142786296
-1.30102999566398 -2.89249122710084
-1.38021124171161 -3.04110619623038
-1.47712125471966 -3.22509594401369
-1.55630250076729 -3.37685252561483
-1.64345267648619 -3.54510679644994
-1.7160033436348 -3.68599765252497
-1.79239168949825 -3.83502105668107
-1.88081359228079 -4.00825259327386
-1.96378782734556 -4.17140698139533
-2.04139268515822 -4.32443741533993
-2.12057393120585 -4.48093810850527
-2.19865708695442 -4.63556874146544
-2.27875360095283 -4.79444586618479
-2.35793484700045 -4.95172350110959
-2.43775056282039 -5.11044288201331
-2.51587384371168 -5.26594423379472
-2.59549622182557 -5.42455488981226
};
\addlegendentry{LI GP}

\end{axis}

\end{tikzpicture}
  } & 
      \resizebox{0.47\textwidth}{!}{
\begin{tikzpicture}

\definecolor{crimson2143940}{RGB}{214,39,40}
\definecolor{darkgray176}{RGB}{176,176,176}
\definecolor{darkorange25512714}{RGB}{255,127,14}
\definecolor{forestgreen4416044}{RGB}{44,160,44}
\definecolor{lightgray204}{RGB}{204,204,204}
\definecolor{steelblue31119180}{RGB}{31,119,180}

\begin{axis}[
legend cell align={left},
legend columns=2,
legend style={
  fill opacity=0.8,
  draw opacity=1,
  text opacity=1,
  at={(0.03,0.97)},
  anchor=north west,
  draw=lightgray204
},
tick align=outside,
tick pos=left,
x grid style={darkgray176},
xlabel={final error (\(\displaystyle \ln_{10}\) scale)},
xmin=-9.45585503542974, xmax=-2.15531381294076,
xtick style={color=black},
y grid style={darkgray176},
ylabel={CPU time (\(\displaystyle \ln_{10}\) scale)},
ymin=-1.36931819391405, ymax=1.16892314852295,
ytick style={color=black}
]
\addplot [semithick, steelblue31119180, mark=*, mark size=2.5, mark options={solid,fill opacity=0}, only marks]
table {%
-2.5692592917292 -0.794747456997227
-2.68979636339957 -0.666424221344731
-2.79567769849672 -0.648642197566046
-2.9752088503594 -0.600104934206612
-3.12393483816584 -0.51446032901348
-3.30801442940528 -0.424902298267025
-3.45981787629456 -0.350337124906865
-3.62810486702793 -0.26682693995314
-3.76901247565403 -0.19679340643113
-3.91804646498948 -0.119776059664185
-4.0912841878078 -0.0332105323095815
-4.25444058970553 0.0637231311140894
-4.407470874872 0.125608149293742
-4.56397022437071 0.20439160976086
-4.71859889178236 0.282460162753743
-4.87747370708313 0.365218719498764
-5.03474899970472 0.451577053739771
-5.19346610184896 0.522206536213787
-5.34896537993161 0.604673171528444
-5.50757411991718 0.678620346279135
};
\addlegendentry{LI UP}
\addplot [semithick, crimson2143940, mark=diamond*, mark size=2.5, mark options={solid,fill opacity=0}, only marks]
table {%
-3.03751436401933 -0.256452560392847
-3.171648743364 -0.205415798154101
-3.287788810815 -0.162583188338139
-3.48179240811714 -0.10269469020786
-3.64025461910744 -0.0398499017742668
-3.83415624593029 0.0301335541572968
-3.99256305416578 0.0969278221891049
-4.16689672713239 0.182115436304375
-4.31201710975371 0.234312521046859
-4.46480801642277 0.311888386642733
-4.64166396249949 0.373964835754808
-4.80762055183278 0.431321738652984
-4.96283568512967 0.485834828084146
-5.12120210497736 0.547991942998046
-5.27737037551255 0.619003326741013
-5.43756408517562 0.693875411328195
-5.59592704979606 0.77238358936755
-5.75555880716131 0.862887740931786
-5.91180559563598 0.931705394151174
-6.07105048095644 1.05354854204854
};
\addlegendentry{C-N}
\addplot [semithick, forestgreen4416044, mark=+, mark size=2.5, mark options={solid,fill opacity=0}, only marks]
table {%
-6.10182585943264 -1.25394358743964
-6.23564898767978 -1.16547767936541
-6.35158368802422 -1.10374344189642
-6.54534931277796 -1.00453960589947
-6.70366419062781 -0.958903400617291
-6.89745287699438 -0.841453307908012
-7.05578841370453 -0.775127962481792
-7.23008793608839 -0.711234341584389
-7.37510157454337 -0.653713022525586
-7.52785839158715 -0.576753954141323
-7.70477601459009 -0.489305930095585
-7.87056834573293 -0.406565290549124
-8.02538563415332 -0.327754535618403
-8.18389079535676 -0.235631191149392
-8.34029072684988 -0.157492469044027
-8.49972311250433 -0.0719218300346352
-8.65787514923195 0.0290085080796776
-8.81718315992945 0.113132388702851
-8.97179362255891 0.183412568188209
-9.12401225258933 0.223277002083061
};
\addlegendentry{Strang}
\addplot [semithick, darkorange25512714, mark=asterisk, mark size=2.5, mark options={solid,fill opacity=0}, only marks]
table {%
-2.48715659578116 -0.695914738859703
-2.60744639819277 -0.633793839811288
-2.71316142786296 -0.584022259974821
-2.89249122710084 -0.48870478397461
-3.04110619623038 -0.412268874627477
-3.22509594401369 -0.302064594879792
-3.37685252561483 -0.174275483723716
-3.54510679644994 -0.113284774441228
-3.68599765252497 -0.0498202369989734
-3.83502105668107 -0.00488094455184998
-4.00825259327386 0.09274378235971
-4.17140698139533 0.170420086253765
-4.32443741533993 0.246050055007829
-4.48093810850527 0.32135930579031
-4.63556874146544 0.40223781892056
-4.79444586618479 0.478318261992494
-4.95172350110959 0.569695255631984
-5.11044288201331 0.656703794222199
-5.26594423379472 0.742856888776406
-5.42455488981226 0.798222642689417
};
\addlegendentry{LI GP}
\end{axis}

\end{tikzpicture}
      }
\end{tabular}
\caption{Comparison for $2^{14}$ points of methods of order 2 applied to the NLH equation \eqref{eq:chaleur} over $(-50,50)$ with Dirichlet boundary conditions.
  On the left panel, final numerical error as a function of the time step (logarithmic scales);
  on the right panel, CPU time (in seconds) as a function of the final numerical error
  (logarithmic scales).}
\label{fig:chaleurordre2}
\end{figure}
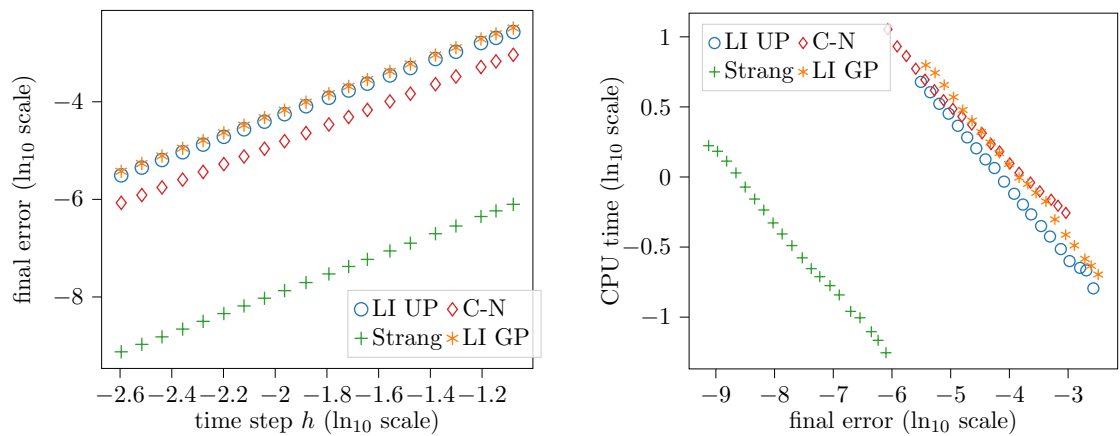

The underlying Runge--Kutta methods of both linearly implicit methods are $\hat{A}$-stable and the eignevalues of their matrix $D$ belong to the open unit disc. Therefore, Theorem \ref{th:convergence} ensures that they converge with order (at least) 2. The numerical experiment described above demonstrates numerically that both methods have order 2. 
For this numerical example, the Strang splitting appears to be the most precise and the most efficient of all four methods. The two linearly implicit methods have an efficiency which is comparable to that of the Crank--Nicolson method. 

\section{Conclusion and future works}
\label{sec:conclusion}

This paper extends the analysis of high order linearly implicit methods developped for ODEs in \cite{DL2020} to the numerical integration of evolution semilinear PDEs. It introduces several notions of stability for the underlying collocation Runge--Kutta method, the analysis of which is carried out in the compagnon paper \cite{DL2023RK}. Assuming these stability conditions as well as additional stability conditions for the extra variables, the main result of this paper states that linearly implicit methode with $s$ stages have order at least $s$ (Theorem \ref{th:convergence}). 

The necessity of all the stability conditions is illustrated numerically. Numerical experiments for the NLS and NLH equations illustrate the efficiency of the linearly implicit methods in dimensions 1 and 2 when compared with classical methods from the litterature. In particular one obtains better performance for high order linearly implicit methods in dimension 2. 

More extensive numerical experiments will be carried out in a forthcoming paper.


\section{Appendix}
\label{sec:appendix}

\subsection{Proof of Proposition \ref{prop:ASIstable}}
\label{subsec:proofASIstable}
This section is devoted to the proof of the estimates of Propositions \ref{prop:ASstable},
\ref{prop:ISstable}, \ref{prop:ASIstable} and \ref{prop:ISIstable}.
For the sake of brevity, we only write the proof of Proposition \ref{prop:ASIstable} in the case
of the torus.
The proof of the three other propositions can be derived using the same technique:
that of \ref{prop:ISIstable} is the very same, with a different localization of the spectrum
of $L$, and that of \ref{prop:ASstable} and \ref{prop:ISstable} has to be adapted for isomophisms
that are not endomorphisms.
That technique has to be adapted for other continuous cases,
for example when $L=\Delta$ or $L=i\Delta$ on $\R^d$, but the spirit is the same.
\begin{proof}
  Let us consider the case of the torus $\Omega=(\R/\Z)^d$, with the Laplace operator $L=\Delta$.
  Let $\sigma>0$ and $h>0$ be fixed. Since $L$ is an unbounded operator from $H^\sigma$ to itself,
  with domain $H^{\sigma+2}$, and is selfadjoint over $H^\sigma$, we have
  \begin{equation*}
    H^{\sigma} = \overline{\underset{\lambda\in {\rm Sp(L)}}{\bigoplus} {\rm Ker} (L-\lambda {\rm Id}_{H^\sigma})}.
  \end{equation*}
  Defining for $\lambda\in {\rm Sp}(L)$, $V_\lambda$ as $\left({\rm Ker}(L-\lambda{\rm Id}_{H^\sigma})\right)^s$, we infer
  \begin{equation}
    \label{eq:densiteHsigma}
    \left(H^\sigma\right)^s = \overline{\underset{\lambda\in {\rm Sp(L)}}{\bigoplus} V_\lambda}.
  \end{equation}
  Observe that, for all $\lambda\in {\rm Sp}(L)$, $V_\lambda$ is stable by $A\otimes L$.
  Moreover, for all $\lambda\in{\rm Sp}(L)$ and $(v_1,\cdots,v_s)\in V_\lambda$, we have
  \begin{equation}
    \label{eq:Vlambdastable}
    \left(A \otimes L\right)
    \begin{pmatrix}
      v_1\\
      \vdots\\
      v_s
    \end{pmatrix}
    = \lambda A
    \begin{pmatrix}
      v_1\\
      \vdots\\
      v_s
    \end{pmatrix}
    .
  \end{equation}
  Observe that ${\rm Sp}(L)\subset (-\infty,0]$.
  Let $\lambda \in {\rm Sp}(L)$ be fixed.
  We have $h\lambda \in \C^-$.
  In particular, with the hypothesis that the method defining $A$ is $ASI$-stable (see Definition
  \ref{def:stab}), the matrix $(I-h\lambda A)$ is invertible.
  This, together with \eqref{eq:Vlambdastable}, implies that
  $({\rm Id}_{(H^\sigma)^s}-h A \otimes L)_{| V_\lambda}$ is invertible from $V_\lambda$ to itself.
  Moreover, using the norm defined in \eqref{eq:norme3}, we have
  \begin{equation*}
    \left\|
      \left[\left({\rm Id}_{(H^\sigma)^s}-h A \otimes L\right)_{| V_\lambda}\right]^{-1}
    \right\|_{(H^\sigma)^s \to (H^\sigma)^s}
    =
    \left\|
      (I-h\lambda A)^{-1}
    \right\|_2,
  \end{equation*}
  where the norm in the right-hand side is the usual hermitian norm on $\C^s$.
  Since this holds for all $\lambda\in{\rm Sp}(L)$ and the spaces $V_{\lambda}$ are pairwise orthogonal,
  we can use \eqref{eq:densiteHsigma}
  to derive that that $({\rm Id}_{(H^\sigma)^s}-h\lambda A \otimes L)$
  is invertible from $(H^\sigma)^s$ to itself, with
  \begin{equation*}
    \left\|
      \left({\rm Id}_{(H^\sigma)^s}-h A \otimes L\right)^{-1}
    \right\|_{(H^\sigma)^s \to (H^\sigma)^s}
    \leq \underset{\lambda\in {\rm Sp}(L)}{\rm sup}
    \left\|
      (I-h\lambda A)^{-1}
    \right\|_2.
  \end{equation*}
  The hypothesis that the method defining $A$ is ASI-stable ensures that the right-hand side
  of the inequality above is bounded independantly of $h>0$.
  This concludes the proof.
\end{proof}

\subsection{Two technical lemmas}
\label{subsec:techlemmas}

\begin{lemma}
  \label{lem:lemmeananummat}
  Let $D\in\mathcal M_s(\C)$. We denote by $\rho(D)$ the spectral radius of the
  matrix $D$ and by $|\cdot|_2$ the usual hermitian norm on $\C^s$.
  For all $\varepsilon>0$, there exists an invertible matrix $P$ such that
  \begin{equation}
    \label{eq:lemmeananummat}
    \forall x\in\C^s,\qquad |P D P^{-1} x |_2 \leq \left(\rho(D)+\varepsilon\right) |x|_2.
  \end{equation}
\end{lemma}

\begin{proof}
  Let $D$ and $\varepsilon$ as in the hypotheses.
  Using for example the Jordan reduction theorem for the matrix $D$,
  there exists an invertible matrix $U\in \mathcal M_s(\C)$ such that $UDU^{-1}$
  is upper triangular. The coefficients on the diagonal of $UDU^{-1}$ are therefore
  the complex eigenvalues $\lambda_1,\dots,\lambda_s$ of $D$.
  After conjugation by the diagonal matrix $M_\delta$ with
  diagonal elements $1,\delta,\cdots,\delta^{s-1}$ for $\delta\in(0,1)$, we obtain
  \begin{equation*}
    M_\delta U D (M_\delta U)^{-1} =
    \begin{pmatrix}
      \lambda_1 & \delta t_{1,2} & \delta^2 t_{1,3} & \dots & \delta^{s-1} t_{1,s} \\
      0 & \lambda_2 & \delta t_{2,1} & \dots & \delta^{s-2} t_{2,s}\\
      \vdots & \ddots & \ddots & \ddots & \vdots\\
      0 & & \ddots & \lambda_{s-1} & \delta t_{s-1,s}\\
      0 & \dots & \dots & 0 &\lambda_s
    \end{pmatrix},
  \end{equation*}
  where the coefficients $(t_{i,j})$ are that of the upper triangular part of $UDU^{-1}$.
  For all vector $u\in\C^s$ with $|u|_2\leq 1$, we infer
  \begin{eqnarray*}
    \lefteqn{|M_\delta U D (M_\delta U)^{-1} u|_2^2}\\
    & = & \sum_{k=1}^s
                                                 \left|\lambda_k u_k + \delta \sum_{j=k+1}^s \delta^{j-k-1} t_{k,j}u_j\right|^2\\
                                           & = & \sum_{k=1}^s \left[
                                                    |\lambda_k u_k|^2
                                                    +2\delta \Re\left(\overline{\lambda_k u_k}
                                                    \sum_{j=k+1}^s \delta^{j-k-1}t_{k,j}u_j\right)
                                                    +\delta^2 \left|\sum_{j=k+1}^s
                                                    \delta^{j-k-1}t_{k,j}u_j\right|^2
                                                    \right]\\
                                           & \leq & \sum_{k=1}^s \rho(D)^2 |u_k|^2 + 2\delta \sum_{k=1}^s \rho(D)|u|_\infty \sum_{j=k+1}^s
                                                    \delta^{j-k-1}|t_{k,j}||u_j|
                                                    + \delta^2
                                                    \sum_{k=1}^s \left|\sum_{j=k+1}^s
                                                    \delta^{j-k-1}|t_{k,j}||u_j|\right|^2\\
  & \leq & \rho(D)^2 + 2 C \delta \rho(D) s |u|_1 + \delta^2 C^2 s |u|_1^2,
  \end{eqnarray*}
  where $C$ denotes the maximum of the moduli of the $(t_{i,j})_{i<j}$ and we have used the fact that
  $|u|_\infty\leq|u|_2\leq 1$. Since we have $|u|_1\leq \sqrt{s}|u|_2\leq \sqrt{s}$, we infer
  finally that
  \begin{equation*}
    |M_\delta U D (M_\delta U)^{-1} u|_2^2 \leq \rho(D)^2 + 2Cs^{3/2}\delta\rho(D)+C^2\delta^2s^2.
  \end{equation*}
  Chosing $\delta\in(0,1)$ small enough to ensure that
  $\delta\left(2Cs^{3/2}\rho(D)+C^2\delta s^2\right)\leq \varepsilon^2$, we infer that
  for all $u\in\C^s$ with $|u|_2\leq 1$,
  \begin{equation*}
    |M_\delta U D (M_\delta U)^{-1} u|_2 \leq \sqrt{\rho(D)^2+\varepsilon^2} \leq \rho(D)+\varepsilon.
  \end{equation*}
  This implies \eqref{eq:lemmeananummat} with $P=M_\delta U$ by homogeneity.
\end{proof}

\begin{lemma}
  \label{lem:normeinduite}
  Let $D\in\mathcal M_s(\C)$. We denote by $\rho(D)$ the spectral radius of the
  matrix $D$.
  For all $\varepsilon>0$, there exists a norm $|\cdot|_D$ on $\C^s$ such that
  \begin{equation}
    \label{eq:lemmenormeinduitte}
    \forall x\in\C^s,\qquad |D x |_D \leq \left(\rho(D)+\varepsilon\right) |x|_D.
  \end{equation}
  In particular, for the norm $\vvvert\cdot\vvvert_D$ on $\mathcal M_s(\C)$ induced by $|\cdot|_D$, we have
  \begin{equation*}
    \vvvert D\vvvert_D \leq \rho(D)+\varepsilon.
  \end{equation*}
\end{lemma}

\begin{proof}
  One obtains the result by setting for all $x\in\C^s$, $|x|_D=|Px|_2$ with the notations of Lemma
  \ref{lem:lemmeananummat}.
\end{proof}

\subsection{Underlying collocation Runge--Kutta methods used for numerical examples}
\label{subsec:examples}

This section is devoted to introducing the collocation Runge--Kutta methods used for numerical examples. Details on the stability properties (see Definition \ref{def:stab}) can be found in \cite{DL2023RK}.

\begin{example}[Collocation Runge--Kutta method with $s=2$ Gauss points]
  \label{ex:meth2etagesGauss}
For $s=2$ Gauss points the method reads,
\begin{equation}
  \label{eq:Gauss2}
\renewcommand\arraystretch{1.2}
\begin{array}
{c|cc}
 \frac{1}{2}-\frac{\sqrt{3}}{6} & \frac{1}{4} & \frac{1}{4}-\frac{\sqrt{3}}{6}\\
\frac{1}{2}+\frac{\sqrt{3}}{6} & \frac{1}{4}+\frac{\sqrt{3}}{6} & 1/4\\
\hline
& \frac{1}{2} & \frac{1}{2}
\end{array},
\end{equation}
  It is $\hat{A}$-stable (see \cite{DL2023RK}).
\end{example}

\begin{example}[Trapezoidal rule with $s=2$ stages]
  \label{ex:meth2etagesLobatto}
  The Runge--Kutta collocation method with $s=2$ and $(c_1,c_2)=(0,1)$ reads
    \[
\renewcommand\arraystretch{1.2}
\begin{array}
{c|cc}
 0 & 0 & 0 \\[1 mm]
 1 & \frac12 & \frac12 \\ [1 mm]
\hline
& \frac12 & \frac12
\end{array}.
\]

  It is $\hat{A}$-stable (see \cite{DL2023RK}).
\end{example}

\begin{example}[Collocation Runge--Kutta method with $s=2$ stages not $A$ nor $I$ stable]\label{ex:notAstable}
The Butcher tableau of the Runge--Kutta collocation method with $s=2$ and $(c_1,c_2)=(1/4,1/3)$ reads
    \[
\renewcommand\arraystretch{1.2}
\begin{array}
{c|cc}
 \frac14 & \frac58 & -\frac38 \\[1 mm]
\frac13 & \frac23 & -\frac13 \\ [1 mm]
\hline
& -2 & 3
\end{array}.
\]

  It is {\bf neither} $A$-stable {\bf nor} $I$-stable (see \cite{DL2023RK}).
\end{example}

\begin{example}[Lobatto collocation method with $s=4$ uniform points]\label{ex:4etages}
The Butcher tableau of the Range-Kutta collocation method with $s=4$ and $(c_1,c_2,c_3,c_4)=(0,\frac13,\frac23,1)$ reads:
\[
\renewcommand\arraystretch{1.2}
\begin{array}
{c|cccc}
0 & 0 & 0 & 0 & 0\\
1/3& 1/8 & 19/72 & -5/72 & 1/72\\
2/3 & 1/9 & 4/9 & 1/9 & 0\\
1& 1/8 & 3/8 & 3/8 & 1/8\\
\hline
& 1/8 & 3/8 & 3/8 & 1/8 
\end{array}.
\]
This method is $\hat{A}$-stable (see \cite{DL2023RK}).

\end{example}

\begin{example}[Collocation Runge--Kutta method with $s=5$ stages which is $A$, $AS$ but not $ASI$ stable]
\label{ex:5etagesAASpasASI}
  We consider the 5-stages Runge--Kutta collocation method defined by its Butcher tableau
  \begin{equation*}
    \renewcommand\arraystretch{1.2}
\begin{array}{c|ccccc}
\frac14 & {\frac{3259}{1440}}&-{\frac{1421}{720}}-
{\frac {21\,\sqrt {7}}{64}}&{\frac{163}{120}}&-{\frac{1421}{720}}+{
\frac {21\,\sqrt {7}}{64}}&{\frac{829}{1440}}\\
\frac12-\frac{\sqrt{7}}{14} & {\medskip}{
\frac { \kappa_{-} ^{2} \left( 281\,\sqrt {7}+1120
 \right) }{15435}}&-{\frac{343}{180}}-{\frac {107\,\sqrt {7}}{315}}&{
\frac { \kappa_{-} ^{2} \left( 106\,\sqrt {7}+455
 \right) }{10290}}&-{\frac { \kappa_{-} ^{2} \left( 
97\,\sqrt {7}+770 \right) }{17640}}&{\frac { \kappa_{-} ^{2} \left( 71\,\sqrt {7}+280 \right) }{15435}}\\
\frac12 & {\medskip}{\frac{203}{90}}&-{\frac{343}{180}}-{\frac {7\,
\sqrt {7}}{24}}&{\frac{22}{15}}&-{\frac{343}{180}}+{\frac {7\,\sqrt {7
                                 }}{24}}&{\frac{53}{90}}\\
\frac12+\frac{\sqrt{7}}{14} & {\medskip}-{\frac { \kappa_{+} ^{2} \left( 281\,\sqrt {7}-1120 \right) }{15435}}&{\frac {
 \kappa_{+} ^{2} \left( 97\,\sqrt {7}-770 \right) }{
17640}}&-{\frac { \kappa_{+} ^{2} \left( 106\,\sqrt {7
}-455 \right) }{10290}}&-{\frac{343}{180}}+{\frac {107\,\sqrt {7}}{315
}}&-{\frac { \kappa_{+} ^{2} \left( 71\,\sqrt {7}-280
    \right) }{15435}}\\
\frac34 & {\medskip}{\frac{363}{160}}&-{\frac{147}{
80}}-{\frac {21\,\sqrt {7}}{64}}&{\frac{63}{40}}&-{\frac{147}{80}}+{
                                                  \frac {21\,\sqrt {7}}{64}}&{\frac{93}{160}}\\
  \hline
  & \frac{128}{45} & \frac{-343}{90} & \frac{44}{15} & \frac{-343}{90} & \frac{128}{45}
\end{array},
\end{equation*}
where $\kappa_{\pm}=7\pm\sqrt{7}$. This method is $A$-stable, $AS$-stable, yet not $ASI$-stable
(see \cite{DL2023RK}). For the linearly implicit method, we choose $\lambda_i=i/6$
for $i\in\{1,\cdots,5\}$. Therefore, we have
\begin{equation*}
  \Theta =
  \left[ \begin {array}{c} \displaystyle{\frac{1153}{1944}}+{\frac {3707\,{\it c_1}}{
1944}}+{\frac {3283\, \left( {\it c_1}-1 \right) ^{2}}{7776}}+{\frac {
61\, \left( {\it c_1}-1 \right) ^{3}}{1944}}+{\frac {5\, \left( {\it c_1
}-1 \right) ^{4}}{7776}}\\ \displaystyle {\medskip}{\frac{1153}{1944}}+{
\frac {3707\,{\it c_2}}{1944}}+{\frac {3283\, \left( {\it c_2}-1
 \right) ^{2}}{7776}}+{\frac {61\, \left( {\it c_2}-1 \right) ^{3}}{
1944}}+{\frac {5\, \left( {\it c_2}-1 \right) ^{4}}{7776}}
\\ \displaystyle {\medskip}{\frac{1153}{1944}}+{\frac {3707\,{\it c_3}}{1944}
}+{\frac {3283\, \left( {\it c_3}-1 \right) ^{2}}{7776}}+{\frac {61\,
 \left( {\it c_3}-1 \right) ^{3}}{1944}}+{\frac {5\, \left( {\it c_3}-1
 \right) ^{4}}{7776}}\\ \displaystyle{\medskip}{\frac{1153}{1944}}+{\frac {
3707\,{\it c_4}}{1944}}+{\frac {3283\, \left( {\it c_4}-1 \right) ^{2}}{
7776}}+{\frac {61\, \left( {\it c4}-1 \right) ^{3}}{1944}}+{\frac {5\,
 \left( {\it c_4}-1 \right) ^{4}}{7776}}\\ \displaystyle{\medskip}{\frac{
1153}{1944}}+{\frac {3707\,{\it c_5}}{1944}}+{\frac {3283\, \left( {
\it c_5}-1 \right) ^{2}}{7776}}+{\frac {61\, \left( {\it c_5}-1 \right) 
^{3}}{1944}}+{\frac {5\, \left( {\it c_5}-1 \right) ^{4}}{7776}}
\end {array} \right] .
\end{equation*}

\end{example}

\begin{example}[Collocation Runge--Kutta with $s=5$ stages which is $I$ but not $A$ stable]
  \label{ex:IstablemaispasAstable}
  We consider the 5-stages Runge--Kutta collocation method defined by its Butcher tableau
  \begin{equation*}
    \renewcommand\arraystretch{1.2}
\begin{array}{c|ccccc}
\frac14 & \frac{4453}{2400} & -\frac{4347}{1600} & \frac{221}{120} & -\frac{1917}{1600} & \frac{1123}{2400}\\
\frac13 & \frac{3824}{2025} & -\frac{133}{50} & \frac{742}{405} & -\frac{179}{150} & \frac{944}{2025}\\
\frac12 & \frac{281}{150} & -\frac{513}{200} & \frac{29}{15} & -\frac{243}{200} & \frac{71}{150}\\
\frac23 & \frac{3808}{2025} & -\frac{194}{75} & \frac{824}{405} & -\frac{28}{25} & \frac{928}{2025}\\
\frac34 & \frac{1503}{800} & -\frac{4131}{1600} & \frac{81}{40} & -\frac{1701}{1600} & \frac{393}{800}\\
  \hline
  & \frac{176}{75} & -\frac{189}{50} & \frac{58}{15} & -\frac{189}{50} & \frac{176}{75} 
\end{array}.
\end{equation*}
This method is $I$-stable but not $A$-stable (see \cite{DL2023RK}).
\end{example}

\section*{Acknowledgments}
This work was partially supported by the Labex CEMPI (ANR-11-LABX-0007-01).

\bibliographystyle{plain}
\bibliography{labib}
\end{document}